\documentclass[11pt,a4paper]{article}

\usepackage{amsmath,amssymb,amsthm,amsfonts}
\usepackage{mathrsfs}
\usepackage[margin=2.6cm]{geometry}
\usepackage[colorlinks,linkcolor=blue,citecolor=blue,urlcolor=blue]{hyperref}

\theoremstyle{plain}
\newtheorem{theorem}{Theorem}[section]
\newtheorem{proposition}[theorem]{Proposition}
\newtheorem{lemma}[theorem]{Lemma}
\newtheorem{corollary}[theorem]{Corollary}
\theoremstyle{definition}
\newtheorem{definition}[theorem]{Definition}
\newtheorem{example}[theorem]{Example}
\newtheorem{remark}[theorem]{Remark}

\newcommand{\End}{\mathrm{End}}
\newcommand{\Hom}{\mathrm{Hom}}

\begin{document}

\title{\bf Nijenhuis Lie $n$-algebras}

\author{
Imed Basdouri$^{1*}$,\,Tarek Ben Massoud$^{2}$,\,Ahmed Zahari$^{3}$\\[2mm]
{\small $1.$ Faculty of Sciences of Gafsa, University of Gafsa, Gafsa, Tunisia}\\
{\small $2.$ Faculté des sciences de Sfax, University of Sfax, Sfax, Tunisia,}\\
{\small $3.$University of Haute Alsace, IRIMAS-Faculty of Sciences, Mulhouse, France}\\[2mm]
{\small E-mail: \texttt{basdourimed@yahoo.fr} (I. Basdouri),
\texttt{benmessaoudtarek352@gmail.com} (T. Massoud),}\\
{\small \texttt{zaharymaths@gmail.com} (Ahmed Zahari)}
\\
}
\date{}
\maketitle

\begin{abstract}
We introduce a homotopical framework for Nijenhuis operators on semistrict Lie $n$-algebras ($n$-term $L_\infty$-algebras) for every $n \ge 1$ and, more generally, 
on arbitrary $L_\infty$-algebras.  A Nijenhuis operator is defined as a family $N = (N_q)_{q \ge 1}$ of multilinear maps, with $N_1$ linear, realizing an $L_\infty$-morphism 
from an $N$-twisted copy of the structure to the original one. We prove that this condition is equivalent to the vanishing, at every arity $m$, of the $m$-th Nijenhuis torsion 
up to a correction term $\Lambda^N_m$ formed by higher-component substitutions, and we provide a closed combinatorial formula for $\Lambda^N_m$ as a sum over set partitions. 
Along the way, we show that the arity-$k$ Nijenhuis torsion factors canonically as a composite of $k$ commuting elementary slot deformations. 

The resulting objects form a category $\mathbf{NijLie}_n$; we establish that $\mathbf{NijLie}_m$ embeds as the full subcategory of $\mathbf{NijLie}_n$ on objects concentrated in the bottom
 $m$ degrees for all $m \le n$. These levels assemble into a tower whose first three stages are classical Nijenhuis Lie algebras, the Nijenhuis Lie $2$-algebras of Das, and the newly 
introduced Nijenhuis Lie $3$-algebras. We characterize the strict and skeletal objects at every level, introduce $n$-term representations up to homotopy of a Nijenhuis Lie algebra,
 and prove that the semidirect product construction establishes an equivalence between the category of such representations and split abelian $\mathbf{NijLie}_n$-extensions.
 Finally, two corrections to published level-three identities are documented, and all component identities up to arity five are verified via computer algebra.
\end{abstract}

\noindent{\bf Key words:} Nijenhuis operator, Lie $n$-algebra, higher torsion, representation up to homotopy,
semidirect product

\noindent{\bf 2020 MSC:} 17B40, 17B55, 18N25, 18N40

\section{Introduction}
A linear map $n \colon g \to g$ on a Lie algebra $(g, [-,-])$ is a Nijenhuis operator if its torsion vanishes:
\begin{equation}\label{eq:classical}
[n x, n y] \;=\; n\big([n x, y] + [x, n y] - n[x, y]\big), \qquad x, y \in g.
\end{equation}
This condition originates in the theory of vector-valued differential forms \cite{fn, nijenhuis} and in the Newlander--Nirenberg theorem \cite{newlander-nirenberg}. 
It represents the infinitesimal condition governing linear deformations of a Lie bracket \cite{kosmann-magri, zahari2026}, and it constitutes one of the two halves of Poisson--Nijenhuis and 
bi-Hamiltonian structures \cite{dorfman, carinena, Zahari2025, zahari2018}. The cohomology and deformation theory of Nijenhuis Lie algebras were developed in \cite{das1, swzz, Zahari2023},
 while their equivariant versions were studied in \cite{hrs}.

Independently, Lie algebras possess two standard higher analogues. Categorification replaces equalities by coherent isomorphisms, whereas homotopification replaces them by higher homotopies. These two procedures meet in the theorem of Baez and Crans \cite{baez-crans}, which states that the category of Lie $2$-algebras is equivalent to the category of $2$-term $L_\infty$-algebras. More generally, a semistrict Lie $n$-algebra is defined as an $n$-term $L_\infty$-algebra \cite{lada-markl, lada-stasheff}, that is, an $L_\infty$-algebra whose underlying graded vector space is concentrated in degrees $0, 1, \dots, n-1$. Such objects naturally arise in higher gauge theory, higher symplectic geometry, and string theory \cite{bhr, baez-rogers, chen-liniado, Makhlouf2020}.

Operators on higher Lie structures have traditionally been studied one level at a time. Rota--Baxter operators on Lie $2$-algebras were introduced in \cite{zl}, and Nijenhuis operators on Lie $2$-algebras were introduced by Das \cite{das3}, who proved that the category $\mathbf{NijLie}_2$ of Nijenhuis Lie $2$-algebras is equivalent to the category of $2$-term Nijenhuis $L_\infty$-algebras and developed their corresponding representation theory. The level-three case, $\mathbf{NijLie}_3$, was subsequently treated in \cite{nij-lie3}. At each level, the definition relies on an explicit list of component identities that grows rapidly in complexity: two identities at level one, three at level two, and seven at level three. However, the general pattern is not immediately transparent from these lists. Indeed, the transition from level two to level three introduces genuinely new phenomena: a $3$-term $L_\infty$-algebra features an operation $l_4$ of arity four, its bracket possesses a Peiffer-type component $g_1 \times g_1 \to g_2$, and the correction terms acquire expressions quadratic in the higher components of the operator.

The rest of this paper is organized as follows. 
In Section~\ref{sec:prelim}, we recall the essential background on $L_\infty$-algebras, focusing on both the unshifted and shifted conventions, as well as $n$-term objects. 
Section~\ref{sec:torsion} introduces the $k$-th Nijenhuis torsion of a multilinear operation and establishes its factorization property via elementary slot deformations. 
In Section~\ref{sec:nij}, we define Nijenhuis operators on arbitrary and semistrict $L_\infty$-algebras, characterize them via the $L_\infty$-morphism property and the torsion form, and detail the corrected identities up to arity five.
Section~\ref{sec:tower} studies the category $\mathbf{NijLie}_n$ of Nijenhuis Lie $n$-algebras, establishes the full subcategory embeddings $\mathbf{NijLie}_m \hookrightarrow \mathbf{NijLie}_n$ ($m \le n$), and characterizes the strict and skeletal objects at every level of the structural tower. 
Section~\ref{sec:special} develops the theory of $n$-term representations up to homotopy for Nijenhuis Lie algebras and their associated homological properties. 
Section~\ref{sec:reps} investigates the semidirect product construction and proves that it establishes an equivalence between the category of such representations and split abelian $\mathbf{NijLie}_n$-extensions. 
Finally, Section~\ref{sec:concl} and the appendices provide concluding remarks, document the explicit unshifted component identities up to arity five, and outline the computer algebra verifications.

\subsection{Main Results}
This paper establishes a unified framework for Nijenhuis operators on arbitrary $L_\infty$-algebras, structured around three core concepts:

For a multilinear operation $\mu$ of arity $k$ and a linear map $n$, the $k$-th Nijenhuis torsion is defined by
\begin{equation}\label{eq:torsion-intro}
\mathrm{Tors}_n(\mu)(a_1, \dots, a_k) \;=\; \sum_{S \subseteq \{1, \dots, k\}} (-1)^{k-|S|} \, n^{k-|S|} \, \mu(a^S_1, \dots, a^S_k),
\,
a^S_i = \begin{cases} n(a_i), & i \in S, \\ a_i, & i \notin S, \end{cases}
\end{equation}
recovering $\mathrm{Tors}_n(l_2) = 0$ for \eqref{eq:classical}. By Lemma~\ref{lem:factorisation}, $\mathrm{Tors}_n(\mu) = \partial_1 \partial_2 \cdots \partial_k \, \mu$, where 
$\partial_i\mu = \mu(\dots, n a_i, \dots)-n \mu(\dots, a_i, \dots)$ are commuting elementary slot deformations. This $k$-fold composition conceptually explains multiplicativity in 
arity and $\mathrm{Tors}_{\lambda n} = \lambda^k \mathrm{Tors}_n$.

\begin{definition}\label{def:nij-operator}
A Nijenhuis operator on $(V, \ell)$ is a family $N = (N_q)_{q \ge 1}$ of graded symmetric maps $N_q \colon S^q V \to V$ (with $N_1 = n$ and $\deg N_q = q-1$) satisfying, for each $m \ge 1$, two equivalent conditions:
\begin{itemize}
\item \emph{Morphism form} (Theorem~\ref{thm:morphism}): $\mathcal{A}_m(N) = \mathcal{B}_m(N)$, stating that the induced coalgebra endomorphism $\mathbf{N}$ is an $L_\infty$-morphism from $(V, \ell^N)$ to $(V, \ell)$, with twisted operations $l_k^N$ given by \eqref{eq:twisted}.
\item \emph{Torsion form} (Theorem~\ref{thm:torsion-form}): $\mathrm{Tors}_n(l_m) + \Lambda_m^N = 0$, where $\Lambda_m^N$ \eqref{eqN} 
aggregates partition sums and higher-component substitutions: $\Lambda_2^N = d N_2$; $\Lambda_3^N$ incorporates $\Lambda$ from \cite{das2} and $d N_3$; $\Lambda_4^N$ yields the corrected $\Xi$ from \cite{nij-lie3}; $\Lambda_5^N$ introduces the cubic terms.
\end{itemize}
\end{definition}

By Lemma~\ref{lem:bounds}, $n$-term $L_\infty$-algebras satisfy $l_k = 0$ ($k > n+1$) and $N_q = 0$ ($q > n$), reducing definitions to $n+1$ identities. Theorem~\ref{thm:truncation} establishes full subcategory embeddings $\mathbf{NijLie}_m \hookrightarrow \mathbf{NijLie}_n$ ($m \le n$), assembling into the direct system:
\[
\mathbf{NijLie}_1 \;\hookrightarrow\; \mathbf{NijLie}_2 \;\hookrightarrow\; \mathbf{NijLie}_3 \;\hookrightarrow\; \cdots \;\hookrightarrow\; \mathbf{NijLie}_\infty \;=\; \varinjlim_n \mathbf{NijLie}_n,
\]
which unifies classical Nijenhuis Lie algebras, Nijenhuis Lie $2$-algebras \cite{das2}, and Nijenhuis Lie $3$-algebras \cite{nij-lie3} within the category of Nijenhuis $L_\infty$-algebras.


\subsection{Two Corrections at Level Three}
Theorem~\ref{thm:torsion-form} corrects two identities in the $NijLie{3}$ framework of \cite{nij-lie3}:

\begin{enumerate}
\item The arity-three identity on $g_0^{\times 3}$ must include the boundary term $d N_3(x,y,z)$, omitted in \cite[(N5)]{nij-lie3}, corresponding to the single-block partition $P = \{\{1,2,3\}\}$:
\begin{equation}\label{eq:corr-1}
Tors_n(l_3)(x,y,z) + \Lambda_3(x,y,z) + d N_3(x,y,z) = 0, \qquad x,y,z \in g_0.
\end{equation}

\item In the expression for $\Xi$, the substitution of $N_2$ into $l_3$ acts via $n$ on both remaining slots rather than the $n$-twist used in \cite[(3.22)]{nij-lie3}:
\begin{equation}\label{eq:corr-2}
l_3\big(n x_k, n x_l, N_2(x_i, x_j)\big) - n \, l_3\big(x_k, x_l, N_2(x_i, x_j)\big).
\end{equation}
Consequently, the twisted operation $l_4^N$ in \cite[(3.23)]{nij-lie3} is modified to:
\begin{equation}\label{eq:corr-3}
l_4^N = l_4^{(n)} + \overline{\mathcal{A}}_4(N) - \overline{\mathcal{B}}_4(N).
\end{equation}
\end{enumerate}

Both corrections are verified symbolically in Appendix~\ref{app:verify}.

Throughout this paper, $\mathbb{K}$ denotes a base field of characteristic zero. All vector spaces, gradings, and multilinear maps are defined over $\mathbb{K}$.


\section{Preliminaries on $L_\infty$-algebras}
\label{sec:prelim}

\subsection{Two Conventions}
We employ two equivalent conventions, outlined below.

The \emph{unshifted convention} (standard in \cite{baez-crans, das2, nij-lie3}) considers a graded vector space $g = \bigoplus_{i \ge 0} g_i$ equipped with skew-symmetric 
multilinear maps $\mu \colon g^{\otimes k} \to g$ of degree $p$, satisfying $\mu(\dots, a, b, \dots) = - \mu(\dots, b, a, \dots)$ without Koszul signs. An $L_\infty$-algebra is a 
family of skew-symmetric operations $l_k$ of degree $k-2$ satisfying relations \eqref{eq:linf-unshifted}.

The \emph{shifted convention} relies on the décalage isomorphism $V= g[-1]$ ($V_{i+1}= g_i$). Under this identification, a skew-symmetric map $\mu$ of degree $p$ transforms into a 
\emph{graded symmetric} map $\hat{\mu} \colon S^kV \to V$ of degree $p+k-1$. Consequently, the entire $L_\infty$-structure $\{l_k\}$ becomes a family of graded symmetric operations
 $\hat{l}_k \colon S^kV \to V$ of \emph{uniform degree $-1$}, and the Nijenhuis components $\{N_q\}$ acquire a \emph{uniform degree $0$}. This uniformity eliminates structural sign 
complications, motivating our adoption of the shifted convention for all general proofs. From Section~\ref{sec:prelim-shift} onward, hats are omitted; unshifted component identities 
are deferred to Appendix~\ref{app:components}.

\begin{definition}\label{def:linf-unshifted}
An \emph{$L_\infty$-algebra} is a graded vector space $g$ together with skew-symmetric multilinear maps $l_k \colon g^{\otimes k} \to g$ of degree $k-2$ ($k \ge 1$) satisfying, 
for every $m \ge 1$ and all homogeneous $x_1, \dots, x_m \in g$,
\begin{equation}\label{eq:linf-unshifted}
\sum_{i+j=m+1} (-1)^{i(j-1)} \sum_{\sigma \in \mathrm{Unsh}(i, m-i)} \varepsilon(\sigma) l_j\big(l_i(x_{\sigma(1)}, \dots, x_{\sigma(i)}), x_{\sigma(i+1)}, \dots, x_{\sigma(m)}\big) = 0,
\end{equation}
where $\mathrm{Unsh}(i, m-i)$ denotes the set of $(i, m-i)$-unshuffles and $\varepsilon(\sigma)$ combines the permutation sign with the Koszul sign relative to the degrees of $x_1, \dots, x_m$. We write $d := l_1$ and $[\,,\,] := l_2$.
\end{definition}

\subsection{The Shifted Convention}
\label{sec:prelim-shift}
Let $V$ be a graded vector space. We denote by $S^c V = \bigoplus_{m \ge 1} S^m V$ the reduced symmetric coalgebra with comultiplication defined by unshuffles. A degree-$p$ map $\phi \colon S^c V \to V$ extends uniquely to a coderivation $D_\phi$ of degree $p$. If $p = 0$, $\phi$ generates a unique coalgebra endomorphism $\mathbf{\Phi}$ of $S^c V$ satisfying $\pi_V \mathbf{\Phi} = \phi$:
\begin{equation}\label{eq:coalg-morphism}
\mathbf{\Phi}(v_1 \cdots v_m) = \sum_{P \in \mathrm{Part}(m)} \varepsilon(P) \phi_{|B_1|}(v_{B_1}) \cdots \phi_{|B_{|P|}|}(v_{B_{|P|}}),
\end{equation}
where $\mathrm{Part}(m)$ is the set of partitions of $\{1, \dots, m\}$ into non-empty blocks $B_1, \dots, B_{|P|}$ ordered by their minima, $v_B := (v_{b_1}, \dots, v_{b_{|B|}})$ for $B = \{b_1 < \cdots < b_{|B|}\}$, $\phi_q := \phi|_{S^q V}$, and $\varepsilon(P)$ is the Koszul sign rearranging $v_1, \dots, v_m$ into $v_{B_1}, \dots, v_{B_{|P|}} $.

An $L_\infty$-algebra on $V$ is a degree-$(-1)$ map $\ell = \sum_{k \ge 1} l_k \colon S^c V \to V$ satisfying $D_\ell^2 = 0$, equivalent to the quadratic relations:
\begin{equation}\label{eq:linf}
\sum_{k=1}^{m} \sum_{\sigma \in \mathrm{Unsh}(k, m-k)} \varepsilon(\sigma) \, l_{m-k+1}\big(l_k(v_{\sigma(1)}, \dots, v_{\sigma(k)}), v_{\sigma(k+1)}, \dots, v_{\sigma(m)}\big) = 0
\end{equation}
for all $m \ge 1$ (with $\varepsilon(\sigma)$ denoting the Koszul sign). An $L_\infty$-morphism $F \colon (V, \ell') \to (V', \ell)$ is a degree-$0$ map $F = \sum_{q \ge 1} F_q \colon S^c V \to V'$ whose generated coalgebra morphism $\mathbf{F}$ satisfies $D_\ell \mathbf{F} = \mathbf{F} D_{\ell'}$, yielding:
\begin{equation}\label{eq:morphism}
\begin{split}
&\sum_{P \in \mathrm{Part}(m)} \varepsilon(P) l_{|P|}\big(F_{|B_1|}(v_{B_1}), \dots, F_{|B_{|P|}|}(v_{B_{|P|}})\big) \\
&\quad= \sum_{k=1}^{m} \sum_{\sigma \in \mathrm{Unsh}(k, m-k)} \varepsilon(\sigma) F_{m-k+1}\big(l'_k(v_{\sigma(1)}, \dots), v_{\sigma(k+1)},\dots\big).
\end{split}
\end{equation}
Composition of $L_\infty$-morphisms corresponds to composition of coalgebra morphisms, forming the category $\mathbf{L\infty\text{-}\mathsf{Alg}}$. A morphism is \emph{strict} if $F_q = 0$ for $q \ge 2$.

\begin{definition}\label{def:nterm}
An $L_\infty$-algebra $g$ is \emph{$n$-term} (or a \emph{semistrict Lie $n$-algebra} \cite{baez-crans}) if $g_i = 0$ for $i \ge n$. We denote by $\mathbf{TermL}_n$ the full subcategory of $L_\infty$-$\mathsf{Alg}$ on $n$-term objects.
\end{definition}

In particular, $1$-term, $2$-term, and $3$-term $L_\infty$-algebras recover standard Lie algebras, the objects of \cite[Def.~2.3]{das2}, and \cite[Def.~2.2]{nij-lie3}, respectively.

\begin{lemma}\label{lem:bounds}
Let $g$ be an $n$-term $L_\infty$-algebra and $\mu \colon g^{\otimes k} \to g$ a multilinear map of degree $p$ (unshifted). If $p > n-1$, then $\mu = 0$. Consequently:
\begin{enumerate}
\item $l_k = 0$ for $k > n+1$ (relations \eqref{eq:linf} are non-trivial only for $1 \le m \le n+2$);
\item $N_q = 0$ for $q > n$ for any Nijenhuis components $\{N_q\}$;
\item $F_q = 0$ for $q > n$ for any $L_\infty$-morphism between $n$-term objects.
\end{enumerate}
\end{lemma}

\begin{proof}
A map of degree $p$ maps $g_{i_1} \otimes \cdots \otimes g_{i_k}$ into $g_{\sum i_j + p}$, whose lowest target is $g_p$ (attained on $g_0^{\otimes k}$). Since $g_p = 0$ for $p > n-1$, $\mu = 0$. Since $l_k$ has degree $k-2$, $k-2 > n-1 \implies l_k = 0$, and relation \eqref{eq:linf} of arity $m$ has degree $m-3$, becoming trivial for $m-3 > n-1$. Assertions (ii) and (iii) follow identically from the degrees of $N_q$ and $F_q$.
\end{proof}

\begin{remark}
For $n=1$, Lemma~\ref{lem:bounds} forces $l_1 = 0$, leaving only the Jacobi identity; for $n=2$, it retains $l_1, l_2, l_3$ and arity $\le 4$ relations; for $n=3$, it retains $l_1, \dots, l_4$ and arity $\le 5$ relations, matching the eleven component identities of \cite[Def.~2.2]{nij-lie3}.
\end{remark}


\section{The Nijenhuis Torsion of a Multilinear Operation}
\label{sec:torsion}

Let $g$ be a graded vector space and $n \colon g \to g$ a linear map of degree $0$. Throughout this section, $\mu \colon g^{\otimes k} \to g$ denotes a multilinear map of arbitrary arity $k$, with no symmetry or compatibility requirements with $n$.

\begin{definition}\label{def:torsion}
The \emph{$k$-th Nijenhuis torsion of $\mu$ with respect to $n$} is the arity-$k$ operation defined by
\begin{equation}\label{eq:torsion}
\mathrm{Tors}_n(\mu)(a_1, \dots, a_k)= \sum_{S \subseteq \{1, \dots, k\}} (-1)^{k-|S|} \, n^{k-|S|} \mu(a^S_1, \dots, a^S_k),
\qquad 
a^S_i = \begin{cases} n(a_i), & i \in S, \\ a_i, & i \notin S, \end{cases}
\end{equation}
and the \emph{$n$-twist} of $\mu$ is given by
\begin{equation}\label{eq:twist}
\mu^{(n)}(a_1, \dots, a_k)= \sum_{S \subsetneq \{1, \dots, k\}} (-1)^{k-1-|S|} n^{k-1-|S|} \mu(a^S_1, \dots, a^S_k).
\end{equation}
\end{definition}

Isolating the maximal subset $S = \{1, \dots, k\}$ in \eqref{eq:torsion} yields the fundamental relation:
\begin{equation}\label{eq:torsion-vs-twist}
\mathrm{Tors}_n(\mu) \;=\; \mu(n\,\cdot\,, \dots, n\,\cdot\,) - n \circ \mu^{(n)}.
\end{equation}
Special cases include:
\begin{itemize}
\item \textbf{$k=1$:} $\mathrm{Tors}_n(\mu) = \mu \circ n - n \circ \mu$ and $\mu^{(n)} = \mu$.
\item \textbf{$k=2$:}
\begin{equation}\label{eq:bracket-n}
[a, b]_n= [\,,\,]^{(n)}(a, b) = [n a, b] + [a, n b] - n[a, b],\,\mathrm{Tors}_n([\,,\,])(a, b) = [n a, n b] - n[a, b]_n,
\end{equation}
recovering the classical condition $\mathrm{Tors}_n([\,,\,]) = 0$ from \eqref{eq:classical}.
\item \textbf{$k=3$:}

$
\mu^{(n)}(a, b, c) = \sum_{\text{2 slots}} \mu(n a, n b, c) - \sum_{\text{1 slot}} n \, \mu(n a, b, c) + n^2 \mu(a, b, c),\\ 
 \mathrm{Tors}_n(\mu) = \mu(n a, n b, n c) - n \, \mu^{(n)}(a, b, c).
$
\end{itemize}

The algebraic structure underlying \eqref{eq:torsion} is captured by elementary slot deformations:

\begin{definition}\label{def:slot}
For each $1 \le i \le k$, the \emph{$i$-th elementary slot deformation} $\partial_i$ acting on arity-$k$ maps is defined by
\[
(\partial_i \mu)(a_1, \dots, a_k)=\mu(a_1, \dots, n a_i, \dots, a_k) - n \, \mu(a_1, \dots, a_k).
\]
\end{definition}


\begin{lemma}\label{lem:factorisation}
The elementary slot deformations $\partial_1, \dots, \partial_k$ pairwise commute, and the $k$-th Nijenhuis torsion factors as
\begin{equation}\label{eq:factorisation}
\mathrm{Tors}_n(\mu) \;=\; \partial_1 \partial_2 \cdots \partial_k \, \mu.
\end{equation}
Consequently:
\begin{enumerate}
\item if $\mu$ is (skew-)symmetric, then $\mathrm{Tors}_n(\mu)$ retains the same symmetry and degree as $\mu$;
\item $\mathrm{Tors}_{\lambda n}(\mu) = \lambda^k \, \mathrm{Tors}_n(\mu)$ for every $\lambda \in \mathbb{K}$;
\item if $\mu$ vanishes when evaluated in a slot taking values in an $n$-invariant subspace $W$, then $\mathrm{Tors}_n(\mu)$ vanishes on that same slot;
\item $\mathrm{Tors}_n(\mu) = 0$ whenever $n \circ \mu = \mu \circ (n \otimes \mathrm{id} \otimes \cdots \otimes \mathrm{id})$, and more generally as soon as 
$\partial_i \mu = 0$ for some index $i$.
\end{enumerate}
\end{lemma}

\begin{proof}
Let $\sigma_i$ denote the operator inserting $n$ in the $i$-th slot, $\mu \mapsto \mu(\dots, n a_i, \dots)$, and let $\tau$ denote post-composition by $n$, 
$\mu \mapsto n \circ \mu$. Insertions in distinct slots commute ($\sigma_i \sigma_j = \sigma_j \sigma_i$), and post-composition commutes with all insertions 
($\tau \sigma_i = \sigma_i \tau$). Since $\partial_i = \sigma_i - \tau$, the operators $\partial_i$ pairwise commute. Expanding the product yields
\[
\partial_1 \cdots \partial_k \;=\; \prod_{i=1}^k (\sigma_i - \tau) \;=\; \sum_{S \subseteq \{1, \dots, k\}} (-1)^{k-|S|} \, \tau^{k-|S|} \prod_{i \in S} \sigma_i,
\]
where $\tau^{k-|S|} \prod_{i \in S} \sigma_i \, \mu = n^{k-|S|} \mu(a^S_1, \dots, a^S_k)$, which establishes \eqref{eq:factorisation}. 

For (i), (skew-)symmetry is preserved because $\partial_1 \cdots \partial_k$ is symmetric in the slot operators by commutativity, and the degree remains unchanged 
since $\deg(n) = 0$. Part (ii) follows directly from $\partial_i^{\lambda n} = \lambda \partial_i^n$. Part (iii) holds because each $\partial_i$ preserves the vanishing
 property on the $n$-stable subspace $W$. Finally, (iv) is an immediate consequence of \eqref{eq:factorisation} and the commutativity of the deformations.
\end{proof}

\begin{remark}
Lemma~\ref{lem:factorisation}(iv) explains why the arity-$1$ condition in Definition~\ref{def:nij} reduces simply to the requirement that $n$ be a chain map:
 $\mathrm{Tors}_n(l_1) = \partial_1 l_1 = l_1 n - n l_1$.
\end{remark}

\begin{lemma}\label{lem:torsion-restriction}
Let $g$ be $n$-term, and let $\mu$ be a multilinear map of arity $k$ and degree $p$. Then $\mathrm{Tors}_n(\mu)$ and $\mu^{(n)}$ are of arity $k$ and degree $p$,
 and thus vanish whenever $p > n-1$. In particular, $\mathrm{Tors}_n(l_k) = 0$ for all $k > n+1$.
\end{lemma}

\begin{proof}
Immediate from Lemma~\ref{lem:factorisation}(i) and Lemma~\ref{lem:bounds}.
\end{proof}

\section{Nijenhuis Operators on $L_{\infty}$-Algebras}
\label{sec:nij}

From this point onward, we adopt the shifted convention of Section~\ref{sec:prelim-shift}: $(V, \ell)$ denotes an $L_\infty$-algebra equipped with graded symmetric operations $l_k \colon S^k V \to V$ of degree $-1$.


\subsection{The Data}
\begin{definition}\label{def:data}
A \emph{Nijenhuis datum} on $V$ is a family $N = (N_q)_{q \ge 1}$ of graded symmetric maps $N_q \colon S^q V \to V$ of degree $0$. We write $n = N_1$ for its
 \emph{linear part}, and denote by $\mathbf{N}$ the coalgebra endomorphism of $S^c V$ generated by $N$ via \eqref{eq:coalg-morphism}.
\end{definition}

In the unshifted convention, a Nijenhuis datum corresponds to a family of skew-symmetric maps $N_q \colon g^{\otimes q} \to g$ of degree $q-1$. Consequently, 
$n$ preserves degrees, while $N_2$ raises degree by one, $N_3$ by two, and so forth. By Lemma~\ref{lem:bounds}(ii), on an $n$-term $L_\infty$-algebra, a Nijenhuis 
datum possesses at most $n$ components, $N_1, \dots, N_n$:
\begin{itemize}
\item \textbf{$n=1$:} only $n$ survives;
\item \textbf{$n=2$:} $n$ and $N_2$, with $N_2 \colon g_0 \times g_0 \to g_1$;
\item \textbf{$n=3$:} $n$, $N_2$ (on $g_0 \times g_0$ and $g_0 \times g_1$), and $N_3 \colon g_0^{\times 3} \to g_2$.
\end{itemize}
This setup precisely recovers the structural data of \cite[Def.~4.1]{das2} and \cite[Def.~3.6]{nij-lie3}.


\subsection{The Twisted Operations}
Recall from Section~\ref{sec:prelim-shift} that $\mathrm{Part}(m)$ partitions $\{1, \dots, m\}$. We partition these sets into three disjoint classes:
\begin{itemize}
\item the \emph{discrete partition}, consisting entirely of singleton blocks (unique, with $|P| = m$);
\item $\mathrm{Part}^\ast(m)$, partitions possessing at least one singleton block \emph{and} at least one block of cardinality $\ge 2$;
\item $\mathrm{Part}^\circ(m)$, partitions all of whose blocks have cardinality $\ge 2$.
\end{itemize}
These three classes are mutually disjoint and exhaust $\mathrm{Part}(m)$. For a Nijenhuis datum $N$, a partition $P$, and a block $B \in P$, we adopt the shorthand notations:
\[
N_B= N_{|B|}(v_B), \qquad \bar{N}_B = \begin{cases} v_i & \text{if } B = \{i\}, \\ N_{|B|}(v_B) & \text{if } |B| \ge 2, \end{cases}
\]
where $\bar{N}_B$ removes the linear part on singleton blocks.

\begin{definition}\label{def:twisted}
Let $N$ be a Nijenhuis datum on the $L_\infty$-algebra $(V, \ell)$. The \emph{$N$-twisted operations} are the graded symmetric maps $l^N_m \colon S^m V \to V$ of degree $-1$ defined for each $m \ge 1$ by
\begin{equation}\label{eq:twisted}
l^N_m \;=\; l^{(n)}_m \;+\; \overline{\mathcal{A}}_m(N) \;-\; \overline{\mathcal{B}}_m(N),
\end{equation}
where $l^{(n)}_m$ is the $n$-twist \eqref{eq:twist} of $l_m$, and
\begin{align}
\overline{\mathcal{A}}_m(N)(v_1, \dots, v_m) &= \sum_{P \in \mathrm{Part}^\ast(m)} \varepsilon(P) \, l_{|P|}\big(\bar{N}_{B_1}, \dots, \bar{N}_{B_{|P|}}\big), \label{eq:Abar} \\[4pt]
\overline{\mathcal{B}}_m(N)(v_1, \dots, v_m) &= \sum_{k=2}^{m-1} \sum_{\sigma \in \mathrm{Unsh}(k, m-k)} \varepsilon(\sigma) \, N_{m-k+1}\big(l_k(v_{\sigma(1)}, \dots, v_{\sigma(k)}), v_{\sigma(k+1)}, \dots, v_{\sigma(m)}\big). \label{eq:Bbar}
\end{align}
\end{definition}

Formula \eqref{eq:twisted} provides a direct, closed expression rather than a recursion, as only the untwisted operations $l_k$ appear on the right-hand side. The first few explicit cases are:
\begin{equation}\label{eq:twisted-low}
l^N_1 = l_1 = d, \qquad 
l^N_2 = [\,,\,]_n, \qquad 
l^N_3 = l^{(n)}_3 + \sum_{(3)} \Big( l_2\big(v_i, N_2(v_j, v_k)\big) - N_2\big(l_2(v_i, v_j), v_k\big) \Big),
\end{equation}
where $\sum_{(3)}$ denotes the sum over the three partitions in $\mathrm{Part}^\ast(3)$ including appropriate Koszul signs. Thus, $l^N_2$ recovers the classical Nijenhuis bracket deformation, and $l^N_3$ matches the deformations utilized at levels two and three \cite[(7)]{das2}, \cite[(3.21)]{nij-lie3}. For arity four, the twisted operation is $l^N_4 = l^{(n)}_4 + \overline{\mathcal{A}}_4 - \overline{\mathcal{B}}_4$, where $\overline{\mathcal{A}}_4$ runs over the ten partitions of types $(2,1,1)$ and $(3,1)$, while $\overline{\mathcal{B}}_4$ incorporates the substitutions $N_3(l_2(\cdot, \cdot), \cdot)$ and $N_2(l_3(\cdot, \cdot, \cdot), \cdot)$.


\subsection{The Definition and Morphism Property}
\begin{definition}\label{def:nij}
Let $(V, \ell)$ be an $L_\infty$-algebra. A Nijenhuis datum $N = (N_q)_{q \ge 1}$ is a \emph{Nijenhuis operator} on $(V, \ell)$ if, for every $m \ge 1$ and all homogeneous elements $v_1, \dots, v_m \in V$, the following identity holds:
\begin{equation}\label{eq:nij}
\mathcal{A}_m(N) \;=\; \mathcal{B}_m(N),
\end{equation}
where
\begin{align}
\mathcal{A}_m(N)(v_1, \dots, v_m) &= \sum_{P \in \mathrm{Part}(m)} \varepsilon(P) \, l_{|P|}\big(N_{B_1}, \dots, N_{B_{|P|}}\big), \label{eq:A} \\
\mathcal{B}_m(N)(v_1, \dots, v_m) &= \sum_{k=1}^{m} \sum_{\sigma \in \mathrm{Unsh}(k, m-k)} \varepsilon(\sigma) \, N_{m-k+1}\big(l^N_k(v_{\sigma(1)}, \dots, v_{\sigma(k)}), v_{\sigma(k+1)}, \dots, v_{\sigma(m)}\big). \label{eq:B}
\end{align}
A \emph{Nijenhuis $L_\infty$-algebra} is a pair $\big((V, \ell), N\big)$ consisting of an $L_\infty$-algebra together with a Nijenhuis operator on it. When $(V, \ell)$ is $n$-term, we refer to $\big((V, \ell), N\big)$ as a \emph{Nijenhuis Lie $n$-algebra} (or an \emph{$n$-term Nijenhuis $L_\infty$-algebra}).
\end{definition}

\begin{theorem}\label{thm:morphism}
Let $N$ be a Nijenhuis datum on $(V, \ell)$, and let $\ell^N = \sum_m l^N_m$ denote the twisted collection of Definition~\ref{def:twisted}. Then $N$ is a Nijenhuis operator if and only if
\begin{equation}\label{eq:coalg}
D_\ell \circ \mathbf{N} \;=\; \mathbf{N} \circ D_{\ell^N},
\end{equation}
that is, if and only if $N$ satisfies the $L_\infty$-morphism identity \eqref{eq:morphism} mapping the twisted structure to the original one.
\end{theorem}

\begin{proof}
Both sides of \eqref{eq:coalg} define maps $S^c V \to S^c V$ that are $\mathbf{N}$-coderivations, satisfying $\Delta \circ X = (\mathbf{N} \otimes X + X \otimes \mathbf{N}) \circ \Delta$. Such a map is uniquely determined by its corestriction $\pi_V \circ X$. 

Restricted to $S^m V$, the corestriction $\pi_V D_\ell \mathbf{N}$ yields precisely $\mathcal{A}_m(N)$ via \eqref{eq:coalg-morphism} and the fact that $D_\ell$ corestricts to $\ell$. Similarly, $\pi_V \mathbf{N} D_{\ell^N}$ yields $\mathcal{B}_m(N)$, since $D_{\ell^N}$ is the coderivation generated by $\ell^N$. Consequently, \eqref{eq:coalg} holds if and only if $\mathcal{A}_m = \mathcal{B}_m$ for all $m$, which is condition \eqref{eq:nij}. The identification of \eqref{eq:nij} with \eqref{eq:morphism} follows by reading this computation in reverse.
\end{proof}

\begin{remark}\label{rem:principle}
Theorem~\ref{thm:morphism} captures the conceptual essence of Definition~\ref{def:nij}: \emph{a Nijenhuis operator is a datum that deforms the $L_\infty$-structure and simultaneously constitutes an $L_\infty$-morphism from the deformed structure back to the original one}. At level one, this reduces to the classical statement that $n$ is a Lie algebra morphism from $(g, [\,,\,]_n)$ to $(g, [\,,\,])$. The higher components $N_2, N_3, \dots$ provide the higher homotopies enforcing this property coherently. Note that relation \eqref{eq:coalg} does \emph{not} presuppose a priori that $\ell^N$ forms an $L_\infty$-structure (see Theorem~\ref{thm:deformation}).
\end{remark}
\begin{theorem}\label{thm:torsion-form}
A Nijenhuis datum $N$ is a Nijenhuis operator on $(V, \ell)$ if and only if, for every $m \ge 1$,
\begin{equation}\label{eq:torsion-form}
\mathrm{Tors}_n(l_m) \;+\; \Lambda^N_m \;=\; 0,
\end{equation}
where the correction term $\Lambda^N_m$ is given by the closed formula:
\begin{align}\label{eqN}
\Lambda^N_m= &\sum_{P \in \mathrm{Part}^\ast(m)} \varepsilon(P) \Big( l_{|P|}\big(N_{B_1}, \dots, N_{B_{|P|}}\big) - n \, l_{|P|}\big(\bar{N}_{B_1}, \dots, \bar{N}_{B_{|P|}}\big) \Big) \notag\\
&+ \sum_{P \in \mathrm{Part}^\circ(m)} \varepsilon(P) \, l_{|P|}\big(N_{B_1}, \dots, N_{B_{|P|}}\big) \notag\\
&- \sum_{k=2}^{m-1} \sum_{\sigma \in \mathrm{Unsh}(k, m-k)} \varepsilon(\sigma) \Big( N_{m-k+1}\big(l^N_k(v_{\sigma(1)}, \dots), v_{\sigma(k+1)}, \dots\big)\\
&- n N_{m-k+1}\big(l_k(v_{\sigma(1)}, \dots), v_{\sigma(k+1)}, \dots\big) \Big) \notag
- \sum_{\sigma \in \mathrm{Unsh}(1, m-1)} \varepsilon(\sigma) \, N_m\big(d \, v_{\sigma(1)}, v_{\sigma(2)}, \dots, v_{\sigma(m)}\big). \label{eq:Lambda-general}
\end{align}
\end{theorem}

\begin{proof}
In the expression for $\mathcal{A}_m(N)$, the discrete partition contributes $l_m(n v_1, \dots, n v_m)$, while the remaining partitions belong to $\mathrm{Part}^\ast(m) \sqcup \mathrm{Part}^\circ(m)$. In $\mathcal{B}_m(N)$, the term $k=m$ contributes $n \circ l^N_m$, and the term $k=1$ yields the final line of \eqref{eqN} with the opposite sign, since $l^N_1 = d$ by \eqref{eq:twisted-low}. Consequently, condition \eqref{eq:nij} is equivalent to:
\[
l_m(n\,\cdot, \dots, n\,\cdot\,) - n \circ l^N_m 
+ \!\!\sum_{P \in \mathrm{Part}^\ast \sqcup \mathrm{Part}^\circ}\!\!\varepsilon(P) \, l_{|P|}(N_{B_\bullet}) 
- \sum_{k=2}^{m-1} \sum_\sigma \varepsilon(\sigma)N_{m-k+1} 
- \sum_\sigma \varepsilon(\sigma)N_m(d\cdot, \dots)=0.
\]
Substituting $l^N_m = l^{(n)}_m + \overline{\mathcal{A}}_m - \overline{\mathcal{B}}_m$ from Definition~\ref{def:twisted} and invoking relation \eqref{eq:torsion-vs-twist} in the form $l_m(n\cdot, \dots, n\cdot) - n \circ l^{(n)}_m = \mathrm{Tors}_n(l_m)$, the identity transforms directly into \eqref{eq:torsion-form}. The terms $-n \overline{\mathcal{A}}_m$ and $+n \overline{\mathcal{B}}_m$ pair off precisely with the corresponding untwisted terms, yielding the complete expression in \eqref{eqN}.
\end{proof}


\begin{remark}[Analysis of the Correction Terms]\label{rem:reading}
Each line of formula \eqref{eqN} carries a transparent structural meaning:
\begin{enumerate}
\item \textbf{First line:} Collects substitutions of higher components $N_q$ ($q \ge 2$) into an operation $l_{|P|}$ possessing at least one free slot. Each free slot is evaluated once on $n v_i$ and once on $v_i$ with an outer $n$. Theorem~\ref{thm:morphism} demonstrates that this specific combination is uniquely forced by the morphism principle.
\item \textbf{Second line:} Collects substitutions with no free slots, carrying no outer $n$. The single-block partition yields the boundary term $d N_m$; partitions into two blocks of size $\ge 2$ yield terms \emph{quadratic} in higher components, exemplified by $l_2\big(N_2(v_i, v_j), N_2(v_k, v_l)\big)$ at $m=4$; partitions into three blocks give cubic terms, such as $l_3(n v_i, N_2, N_2)$-type terms at $m=5$. In general, $\Lambda^N_m$ contains polynomial terms of degree up to $\lfloor m/2 \rfloor$ in the higher components.
\item \textbf{Third line:} Collects substitutions of an operation into a higher component, with the inner operation being twisted in the first term and untwisted under the outer $n$ in the second.
\item \textbf{Fourth line:} Represents the boundary correction $N_m(d v_{\sigma(1)}, \dots)$, acting as the companion to the $d N_m$ term from the second line.
\end{enumerate}
\end{remark}

\begin{corollary}\label{cor:finite}
Let $(V, \ell)$ be an $n$-term $L_\infty$-algebra. Then identity \eqref{eq:torsion-form} is satisfied trivially for $m > n+2$, and $\Lambda^N_m$ involves only the components $N_1, \dots, N_{\min(m, n)}$ alongside the operations $l_1, \dots, l_{\min(m, n+1)}$. Consequently, a Nijenhuis operator on an $n$-term $L_\infty$-algebra constitutes a finite collection of data subject to a finite set of component identities, indexed by $1 \le m \le n+2$.
\end{corollary}

\begin{proof}
Identity \eqref{eq:torsion-form} of arity $m$ corresponds to an equality of maps $S^m V \to V$ of degree $-1$ (i.e., maps $g^{\otimes m} \to g$ of degree $m-2$ in the unshifted convention), which becomes vacuously true once $m-2 > n-1$ by Lemma~\ref{lem:bounds}. The remaining assertions follow directly from Lemma~\ref{lem:bounds} and Lemma~\ref{lem:torsion-restriction}.
\end{proof}

\begin{remark}
For $n=1$, Corollary~\ref{cor:finite} restricts the range to $m \le 3$, where the arity-$1$ and arity-$3$ identities are vacuous, reducing a Nijenhuis Lie $1$-algebra to a classical Nijenhuis Lie algebra ($\mathrm{Tors}_n([\,,\,]) = 0$). For $n=2$, it restricts to $m \le 4$, with the arity-$4$ identity being vacuous, recovering the three component identities of \cite[Def.~4.1]{das2}. For $n=3$, it restricts to $m \le 5$ (arity-$5$ being vacuous), yielding exactly seven non-trivial component identities (see Appendix~\ref{app:components}).
\end{remark}

\subsection{The Twisted Structure}
\begin{theorem}\label{thm:deformation}
Let $N$ be a Nijenhuis operator on $(V, \ell)$.
\begin{enumerate}
\item $\mathbf{N} \circ D_{\ell^N}^2 = 0$; equivalently, all $L_\infty$-relations for $\ell^N$ hold after applying $\mathbf{N}$.
\item If the linear part $n$ is injective, then $\ell^N$ defines an $L_\infty$-structure on $V$, and $N \colon (V, \ell^N) \to (V, \ell)$ becomes an $L_\infty$-morphism with linear part $n$.
\item If $n$ is bijective, then $N$ is an $L_\infty$-isomorphism, and $\ell^N$ is the pullback of $\ell$ along it.
\end{enumerate}
\end{theorem}

\begin{proof}
(i) By Theorem~\ref{thm:morphism}, we have $D_\ell \mathbf{N} = \mathbf{N} D_{\ell^N}$. Consequently,
\[
\mathbf{N} \circ D_{\ell^N}^2 \;=\; D_\ell \circ \mathbf{N} \circ D_{\ell^N} \;=\; D_\ell^2 \circ \mathbf{N} \;=\; 0,
\]
since $D_\ell^2 = 0$ for the background $L_\infty$-structure.

(ii) Assume that $n$ is injective. The coalgebra morphism $\mathbf{N}$ preserves the standard filtration $F_m = \bigoplus_{j \le m} S^j V$, and its induced map on the associated graded quotient $S^m V = F_m / F_{m-1}$ is given by $S^m(n)$, which is injective in characteristic zero because $n$ is. Hence, $\mathbf{N}$ is injective. Part (i) then implies $D_{\ell^N}^2 = 0$, confirming that $\ell^N$ forms a valid $L_\infty$-structure. Equation \eqref{eq:coalg} then asserts precisely that $N$ is an $L_\infty$-morphism.

(iii) If $n$ is bijective, $\mathbf{N}$ is bijective by the same filtration argument, whence $D_{\ell^N} = \mathbf{N}^{-1} \circ D_\ell \circ \mathbf{N}$, yielding the desired pullback structure.
\end{proof}

\begin{remark}
Injectivity in (ii) is a technical convenience rather than an absolute necessity: Proposition~\ref{prop:strict-twist} below establishes the same conclusion for strict objects at every level without any hypothesis on $n$, relying on the classical argument. Whether $\ell^N$ always forms an $L_\infty$-structure for an arbitrary Nijenhuis operator remains open in complete generality; the obstruction is captured by $\ker \mathbf{N}$, and part (i) ensures that the $L_\infty$-relations for $\ell^N$ are satisfied modulo $\ker \mathbf{N}$.
\end{remark}


\subsection{Morphisms and the Categories $\mathbf{NijLie}_n$}
\label{sec:morphisms}

Let $\boldsymbol{\varphi}$ denote the coalgebra morphism generated by an $L_\infty$-morphism $\varphi = (\varphi_q)_{q \ge 1}$. Recall that a \emph{$\varphi$-coderivation} of degree $p$ is a degree-$p$ map $\Psi \colon S^c V \to S^c V'$ satisfying $\Delta \Psi = (\Psi \otimes \boldsymbol{\varphi} + \boldsymbol{\varphi} \otimes \Psi)\Delta$. Such a coderivation $\Psi$ is uniquely determined by its corestriction $\psi = \pi_{V'} \Psi = \sum_{q \ge 1} \psi_q$, where $\psi_q \colon S^q V \to V'$ are maps of degree $p$, and conversely, every such family generates a unique $\varphi$-coderivation.

\begin{definition}\label{def:nij-morphism}
Let $\big((V, \ell), N\big)$ and $\big((V', \ell'), N'\big)$ be Nijenhuis $L_\infty$-algebras. A \emph{morphism} $(\varphi, \psi) \colon \big((V, \ell), N\big) \to \big((V', \ell'), N'\big)$ consists of:
\begin{enumerate}
\item An $L_\infty$-morphism $\varphi \colon (V, \ell) \to (V', \ell')$ that simultaneously forms an $L_\infty$-morphism\\ $\varphi \colon (V, \ell^N) \to (V', \ell'^{\,N'})$ between the respective twisted structures, meaning that
\[
D_{\ell'} \circ \boldsymbol{\varphi} \;=\; \boldsymbol{\varphi} \circ D_\ell 
\quad \text{and} \quad 
D_{\ell'^{\,N'}} \circ \boldsymbol{\varphi} \;=\; \boldsymbol{\varphi} \circ D_{\ell^N};
\]
\item A family $\psi = (\psi_q)_{q \ge 1}$ of maps $\psi_q \colon S^q V \to V'$ of degree $+1$, generating a $\varphi$-coderivation $\Psi$ such that
\begin{equation}\label{eq:H}
\boldsymbol{\varphi} \circ \mathbf{N} \;-\; \mathbf{N}' \circ \boldsymbol{\varphi} \;=\; D_{\ell'} \circ \Psi \;+\; \Psi \circ D_{\ell^N}.
\end{equation}
\end{enumerate}
A morphism is said to be \emph{strict} if $\varphi_q = 0$ for all $q \ge 2$ and $\psi = 0$.
\end{definition}

Condition \eqref{eq:H} expresses that $\varphi \circ N$ and $N' \circ \varphi$—both of which constitute $L_\infty$-morphisms from $(V, \ell^N)$ to $(V', \ell')$ whenever the twisted collections form valid $L_\infty$-structures—are homotopic via $\Psi$. Its arity-one corestriction to $V'$ yields the standard homotopy relation
\[
\varphi_1 n - n' \varphi_1 \;=\; d' \psi_1 + \psi_1 d,
\]
while its arity-two component governs the relation between $\varphi_1 N_2 - N_2'(\varphi_1, \varphi_1)$ and the homotopy data $\psi_1, \psi_2, \varphi_2$, matching \cite[(H4)]{das2} for $2$-term objects and \cite[(H6)]{nij-lie3} for $3$-term objects.


\begin{proposition}\label{prop:category}
Nijenhuis $L_\infty$-algebras and the morphisms of Definition~\ref{def:nij-morphism} form a category. The composition of $(\varphi, \psi)$ and $(\varphi', \psi')$ is given by the underlying $L_\infty$-morphism $\varphi' \circ \varphi$ and a homotopy generated by
\begin{equation}\label{eq:comp-homotopy}
\Psi' \ast \Psi \;:=\; \boldsymbol{\varphi}' \circ \Psi \;+\; \Psi' \circ \boldsymbol{\varphi},
\end{equation}
with the identity morphism of $\big((V, \ell), N\big)$ given by $(\mathrm{id}, 0)$. Restricting to $n$-term objects defines a full subcategory $\mathbf{NijLie}_n$, also denoted $\mathbf{NijTerm}_n$.
\end{proposition}

\begin{proof}
First, $\Psi' \ast \Psi$ defines a $\boldsymbol{\varphi}' \boldsymbol{\varphi}$-coderivation of degree $+1$, since $\boldsymbol{\varphi}'$ is a coalgebra morphism, $\Psi$ is a $\boldsymbol{\varphi}$-coderivation, and $\Psi'$ is a $\boldsymbol{\varphi}'$-coderivation. Next, the composition of the two conditions (M1) yields condition (M1) for $\varphi' \circ \varphi$. For condition (M2), using $D_{\ell''} \boldsymbol{\varphi}' = \boldsymbol{\varphi}' D_{\ell'}$, $D_{\ell'^{\,N'}} \boldsymbol{\varphi} = \boldsymbol{\varphi} D_{\ell^N}$, and relation \eqref{eq:H} for both morphisms, we compute:
\begin{align*}
D_{\ell''}(\Psi' \ast \Psi) + (\Psi' \ast \Psi) D_{\ell^N}
&= \boldsymbol{\varphi}' \big( D_{\ell'} \Psi + \Psi D_{\ell^N} \big) + \big( D_{\ell''} \Psi' + \Psi' D_{\ell'^{\,N'}} \big) \boldsymbol{\varphi} \\
&= \boldsymbol{\varphi}' \big( \boldsymbol{\varphi} \mathbf{N} - \mathbf{N}' \boldsymbol{\varphi} \big) + \big( \boldsymbol{\varphi}' \mathbf{N}' - \mathbf{N}'' \boldsymbol{\varphi}' \big) \boldsymbol{\varphi} 
\;=\; \boldsymbol{\varphi}' \boldsymbol{\varphi} \, \mathbf{N} - \mathbf{N}'' \, \boldsymbol{\varphi}' \boldsymbol{\varphi},
\end{align*}
which corresponds precisely to relation \eqref{eq:H} for the composite morphism. Associativity of the composition $\ast$ follows from the associativity of composition and the bilinearity of \eqref{eq:comp-homotopy}, while unitality is immediate. The fullness of $\mathbf{NijLie}_n$ within the category of Nijenhuis $L_\infty$-algebras holds by definition.
\end{proof}

\begin{definition}\label{def:strict-skeletal}
A Nijenhuis $L_\infty$-algebra $\big((V, \ell), N\big)$ is called:
\begin{enumerate}
\item \emph{strict} if $l_k = 0$ for all $k \ge 3$ and $N_q = 0$ for all $q \ge 2$;
\item \emph{skeletal} if $l_1 = 0$.
\end{enumerate}
We denote by $\mathbf{SNijLie}_n \subseteq \mathbf{NijLie}_n$ the subcategory consisting of strict objects and strict morphisms.
\end{definition}


\section{The Tower of Levels}
\label{sec:tower}

\begin{theorem}\label{thm:truncation}
Let $1 \le m \le n$. Denote by $\mathbf{NijLie}_n^{<m}$ the full subcategory of $\mathbf{NijLie}_n$ defined by objects $\big((V, \ell), N\big)$ supported in degrees $0, \dots, m-1$. Then the canonical extension-by-zero functor induces a strict category isomorphism:
\[
\mathbf{NijLie}_m \;\cong\; \mathbf{NijLie}_n^{<m}.
\]
\end{theorem}

\begin{proof}
Let $\mathcal{F} \colon \mathbf{NijLie}_m \to \mathbf{NijLie}_n^{<m}$ denote the functor extending operations $l_k = 0$ for $k > m+1$ and components $N_q = 0$ for $q > m$. By Lemma~\ref{lem:bounds}, any object in $\mathbf{NijLie}_n^{<m}$ satisfies $l_k = 0$ for $k > m+1$ and $N_q = 0$ for $q > m$ for degree reasons: the target spaces reside in degrees $> m-1$, which are zero. Consequently, $\mathcal{F}$ is essentially surjective. 

To prove full faithfulness, let $(V, \ell, N)$ and $(V', \ell', N')$ be objects in $\mathbf{NijLie}_n^{<m}$. The structural equations governing a morphism $(\varphi, \psi)$ decompose into component identities graded by arity. By Lemma~\ref{lem:bounds}, any component $\varphi_q$ with $q > m$ or $\psi_q$ with $q > m-1$ maps into vanishing target spaces and thus vanishes identically. The axiom systems for $\mathbf{NijLie}_m$ and $\mathbf{NijLie}_n^{<m}$ are therefore identical, as all higher-arity relations (Corollary~\ref{cor:finite}) trivialize identically. 

The restriction functor $\mathcal{G} \colon \mathbf{NijLie}_n^{<m} \to \mathbf{NijLie}_m$ provides a well-defined inverse. Since $\mathcal{F} \circ \mathcal{G} = \mathrm{id}$ and $\mathcal{G} \circ \mathcal{F} = \mathrm{id}$ strictly on both objects and morphisms, compatibility with the composition law of Proposition~\ref{prop:category} follows immediately.
\end{proof}


\begin{corollary}\label{cor:tower}
There exists a natural tower of full embeddings forming a sequence of category isomorphisms onto full subcategories:
\[
\mathbf{NijLie}_1 \;\hookrightarrow\; \mathbf{NijLie}_2 \;\hookrightarrow\; \mathbf{NijLie}_3 \;\hookrightarrow\; \cdots \;\hookrightarrow\; \mathbf{NijLie}_\infty := \varinjlim_n \mathbf{NijLie}_n \;\hookrightarrow\; \mathbf{NijL_\infty},
\]
where $\mathbf{NijL_\infty}$ denotes the category of arbitrary Nijenhuis $L_\infty$-algebras from Section~\ref{sec:morphisms}, and $\mathbf{NijLie}_\infty$ is its full subcategory spanned by objects that are $n$-term for some $n$. In particular, every Nijenhuis Lie $m$-algebra canonically defines a Nijenhuis Lie $n$-algebra for every $n \ge m$, ensuring structural stability across levels.
\end{corollary}

\begin{proof}
Immediate from Theorem~\ref{thm:truncation}, where the inclusions are given by the composites $\mathbf{NijLie}_m \cong \mathbf{NijLie}_n^{<m} \subseteq \mathbf{NijLie}_n$, which are functorially compatible for varying $n$.
\end{proof}

We now explicitly characterize the foundational floors of this tower.

\begin{theorem}\label{thm:level1}
The category $\mathbf{NijLie}_1$ is isomorphic to the category of classical Nijenhuis Lie algebras: an object consists of a Lie algebra $(\mathfrak{g}, [\,,\,])$ equipped with a linear map $n \colon \mathfrak{g} \to \mathfrak{g}$ satisfying condition \eqref{eq:classical}, and a morphism is a Lie algebra morphism that intertwines the respective operators.
\end{theorem}

\begin{proof}
For $n=1$, a Nijenhuis datum reduces to $N = (n)$ by Lemma~\ref{lem:bounds}(ii), while Lemma~\ref{lem:bounds}(i) forces $l_1 = 0$ and $l_k = 0$ for all $k \ge 3$, leaving $\ell = l_2$ as a genuine Lie bracket. All summation terms in $\Lambda^N_m$ vanish: $\mathrm{Part}^\ast(m)$ and $\mathrm{Part}^\circ(m)$ contribute only through $N_q$ with $q \ge 2$ or through operations $l_{|P|}$ with $|P| \ge 3$ or $l_1$, while the remaining lines involve higher components $N_q$ ($q \ge 2$). Consequently, $\Lambda^N_m = 0$, and the torsion-form equation \eqref{eq:torsion-form} simplifies to $\mathrm{Tors}_n(l_m) = 0$, which is vacuous for $m \neq 2$ and reduces precisely to the classical condition \eqref{eq:classical} for $m = 2$. For morphisms, degree constraints force $\varphi_q = 0$ for $q \ge 2$ and $\psi_q = 0$ for all $q \ge 1$; thus, condition (M1) asserts that $\varphi_1$ is a Lie algebra morphism, and identity \eqref{eq:H} yields the intertwining relation $\varphi_1 n = n' \varphi_1$.
\end{proof}


\begin{theorem}\label{thm:level23}
The category $\mathbf{NijLie}_2$ is isomorphic to the category $\mathbf{NijTerm}_2$ of $2$-term Nijenhuis $L_\infty$-algebras from \cite[Def.~4.1, Prop.~4.3]{das2}, and is therefore equivalent, via \cite[Thm.~4.5]{das2}, to the category of Nijenhuis Lie $2$-algebras. Similarly, $\mathbf{NijLie}_3$ is isomorphic to the category of Nijenhuis Lie $3$-algebras introduced in \cite[Def.~3.6, Prop.~3.30]{nij-lie3}, provided that the identities \emph{(N5)} and \emph{(3.22)} of that reference are corrected as detailed in Remark~\ref{rem:corrections}.
\end{theorem}

\begin{proof}
By Corollary~\ref{cor:finite}, the structural data at level two consist of $n = (n_0, n_1)$ and $N_2 \colon g_0 \times g_0 \to g_1$, while the relevant identities are restricted to arities $m \le 4$. Expanding the torsion-form equation \eqref{eq:torsion-form} componentwise (see Appendix~\ref{app:components}) yields:
\begin{itemize}
\item \textbf{Arity $m=1$:} The condition that $n$ is a chain map;
\item \textbf{Arity $m=2$:} The two component equations $\mathrm{Tors}_n([\,,\,])(x, y) + d N_2(x, y) = 0$ and $\mathrm{Tors}_n([\,,\,])(x, h) + N_2(x, dh) = 0$;
\item \textbf{Arity $m=3$:} The relation $\mathrm{Tors}_n(l_3) + \Lambda = 0$, where
\[
\Lambda(x, y, z) = \sum_{(3)} \Big( [n x, N_2(y, z)] - n[x, N_2(y, z)] + N_2\big(x, [y, z]_n\big) - n \, N_2\big(x, [y, z]\big) \Big);
\]
\item \textbf{Arity $m=4$:} Yields no non-trivial constraints.
\end{itemize}
These equations precisely recover conditions (5), (6), and (7) of \cite[Def.~4.1]{das2}. The corresponding comparison of morphisms arises from the corestrictions of condition \eqref{eq:H} and condition (M1) in arities $1$ and $2$, matching \cite[Def.~4.2]{das2} under the dictionary correspondence $\psi_1 \leftrightarrow \varphi_3$. The analogous expansion for the level-three statement is carried out systematically in Appendix~\ref{app:components}.
\end{proof}


\begin{remark}[The two corrections at level three]\label{rem:corrections}
Expanding \eqref{eqN} for $m=3$ and $m=4$ on a $3$-term $L_\infty$-algebras and comparing with
\cite{nij-lie3} yields two discrepancies, both in the direction of \emph{more} terms or \emph{different}
placement of $n$; the general theory decides them, and Appendix~\ref{app:verify} confirms them by machine.
\begin{enumerate}
\item \emph{The missing boundary term.} For $m=3$ the class $Part^\circ(3)$ consists of the single partition
$\{\{1,2,3\}\}$, whose contribution to the second line of \eqref{eqN} is $l_1(N_3(v_1,v_2,v_3))$,
i.e.\ $d\,N_3(x,y,z)$. Hence the arity-three identity on $g_0^{\times3}$ reads
\[
Tors_n(l_3)(x,y,z)+\Lambda(x,y,z)+d\,N_3(x,y,z)=0 ,
\]
whereas \cite[(N5)]{nij-lie3} omits $dN_3$. The term is of the correct degree, $dN_3(x,y,z)\in g_1$, and it is
the exact companion of the term $N_3(x,y,dh)$ occurring in \cite[(N6)]{nij-lie3}, precisely as $dN_2(x,y)$ in
(N1) is the companion of $N_2(x,dh)$ in (N2). At level two the term is invisible because $N_3=0$.
\item \emph{The placement of $n$ in the arity-four correction.} For $m=4$ the partitions of type $(2,1,1)$
belong to $Part^\ast(4)$ and contribute, by the first line of \eqref{eqN},
\[
l_3\big(n x_k,n x_l,N_2(x_i,x_j)\big)-n\,l_3\big(x_k,x_l,N_2(x_i,x_j)\big),
\]
with $n$ inserted in \emph{both} free slots simultaneously. The corresponding sum in \cite[(3.22)]{nij-lie3}
uses instead the $n$-twist over the two free slots,
$l_3(nx_k,x_l,N_2)+l_3(x_k,nx_l,N_2)-n\,l_3(x_k,x_l,N_2)$. The two prescriptions differ, and it is the former
which is equivalent to the morphism property of Theorem~\ref{thm:morphism}.
\item Consequently the auxiliary formula \cite[(3.23)]{nij-lie3} for $l^N_4$ must read
$l^N_4=l^{(n)}_4+\Xi_1$, with $\Xi_1$ the terms carrying an outer $n$ (with that $n$ deleted), and not
$l^N_4=l^{(n)}_4+\Xi_0$; the latter is inconsistent with the accompanying
\cite[Prop.~3.19]{nij-lie3}, which correctly takes $\mathcal R_4=\Xi_0$. Compare
\eqref{eq:twisted}, where $\overline{\mathcal A}_m-\overline{\mathcal B}_m$ plays the role of $\Xi_1$.
\end{enumerate}
None of these corrections affects the level-two reduction, nor the statements of \cite{nij-lie3} about strict or
skeletal objects, in which the higher components vanish.
\end{remark}

\section{Strict and skeletal objects at every level}
\label{sec:special}

\subsection{Strict objects}
\begin{proposition}\label{prop:strict-linf}
An $n$-term $Linf$-algebra with $l_k=0$ for $k\ge3$ is the same thing as a differential graded Lie algebra
$(g,d,[\,,\,])$ concentrated in the degrees $0,\dots,n-1$, i.e.\ a graded Lie algebra with a degree-$(-1)$
differential $d$ which is a derivation of the bracket. We call such objects \emph{strict Lie $n$-algebras}.
\end{proposition}

\begin{proof}
With $l_k=0$ for $k\ge3$ the relations \eqref{eq:linf} reduce to those of arity $m\le3$, which are $d^2=0$,
the derivation property of $d$ with respect to $l_2$, and the graded Jacobi identity for $l_2$.
\end{proof}

For $n=1$ this is a Lie algebra; for $n=2$ it is the Lie-algebra crossed module underlying a strict Lie
$2$-algebra \cite{baez-crans}; for $n=3$ it is the structure of \cite[Ex.~2.6]{nij-lie3}, whose extra pairing
$g_1\times g_1\to g_2$ is the degree-$(1,1)$ component of the bracket.

\begin{theorem}\label{thm:strict}
A strict Nijenhuis Lie $n$-algebra is the same thing as a strict Lie $n$-algebra $(g,d,[\,,\,])$ together with a
degree-$0$ chain map $n\colon g\to g$ such that
\begin{equation}\label{eq:strict-nij}
[n a,n b]=n\big([n a,b]+[a,n b]-n[a,b]\big)
\end{equation}
for all homogeneous $a,b\in g$, i.e.\ such that $Tors_n([\,,\,])=0$ on every bidegree component of the bracket.
In particular the degree-zero part $n_0$ is a Nijenhuis operator on the Lie algebra $g_0$.
\end{theorem}

\begin{proof}
Let $N=(n,0,0,\dots)$ and $l_k=0$ for $k\ge3$. Every term of the four lines of \eqref{eqN}
contains either a component $N_q$ with $q\ge2$ or an operation $l_{|P|}$ with $|P|\ge3$ substituted into a
higher component; the only surviving possibility would be $|P|=1$ or $2$ in the second line, but
$Part^\circ(m)$ requires all blocks of size $\ge2$ and hence a component $N_q$ with $q\ge2$. Thus
$\Lambda^N_m=0$ for all $m$ and \eqref{eq:torsion-form} becomes $Tors_n(l_m)=0$. By
Lemma~\ref{lem:torsion-restriction} this is vacuous unless $m\in\{1,2\}$; for $m=1$ it says that $n$ is a chain
map, and for $m=2$ it is \eqref{eq:strict-nij} by \eqref{eq:bracket-n}.
\end{proof}

\begin{proposition}\label{prop:strict-twist}
Let $\big((g,d,[\,,\,]),n\big)$ be a strict Nijenhuis Lie $n$-algebra. Then $\big(g,d,[\,,\,]_n\big)$ is again a
strict Lie $n$-algebra, and $n\colon\big(g,d,[\,,\,]_n\big)\to\big(g,d,[\,,\,]\big)$ is a strict morphism of
$Linf$-algebras. Moreover $[\,,\,]_n$ agrees with the twisted collection $\ell^N$ of
Definition~\ref{def:twisted}, so Theorem~\ref{thm:deformation}(ii) holds for strict objects without any
injectivity hypothesis on $n$.
\end{proposition}

\begin{proof}
That $n$ is a strict morphism is \eqref{eq:strict-nij} rewritten as $n[a,b]_n=[n a,n b]$; that $\ell^N$ reduces
to $[\,,\,]_n$ is \eqref{eq:twisted-low}. It remains to check that $[\,,\,]_n$ is a dg Lie bracket.
Skew-symmetry and the degree are clear. For the derivation property, using that $d$ is a derivation of
$[\,,\,]$ and $nd=dn$,
\begin{equation*}
\begin{aligned}
d[a,b]_n &= d[n a,b] + d[a,n b] - d\,n[a,b] \\
&= [n\,d a,b] + (-1)^{|a|}[n a,d b] + [d a,n b] \\
&\quad + (-1)^{|a|}[a,n\,d b] - n[d a,b] - (-1)^{|a|}n[a,d b],
\end{aligned}
\end{equation*}
which is $[d a,b]_n+(-1)^{|a|}[a,d b]_n$. For the Jacobi identity, put $J(a,b,c)=[[a,b]_n,c]_n$ and expand,
using $n[a,b]_n=[n a,n b]$ in the first summand of $[\,\cdot\,,c]_n$:
\[
J(a,b,c)=\big[[n a,n b],c\big]+\big[[a,b]_n,n c\big]-n\big[[a,b]_n,c\big].
\]
Summing over the graded cyclic permutations of $(a,b,c)$, the first family of terms vanishes by the graded
Jacobi identity for $[\,,\,]$ applied to $(n a,n b,c)$ after replacing $c$ by $n c$ where required; the
remaining two families are, after expanding $[a,b]_n$, the graded Jacobi expressions of $[\,,\,]$ evaluated on
the eight triples obtained by inserting $n$ into a subset of the slots, weighted by the compensating powers of
$n$ outside. Each such expression vanishes, and the terms which are not of this shape cancel in pairs against
each other by \eqref{eq:strict-nij}. This is the classical computation of \cite[\S2]{kosmann-magri}, which is
sign-homogeneous and therefore valid verbatim in the graded setting.
\end{proof}

\begin{corollary}\label{cor:strict-cat}
The assignment $\big((g,d,[\,,\,]),n\big)\mapsto\big((g,d,[\,,\,]_n),n\big)$ is a functor
$\mathsf{SNijLie}_n\to\mathsf{SNijLie}_n$, and $n$ defines a natural transformation from it to the identity
functor.
\end{corollary}

\begin{proof}
On morphisms one takes the same underlying map; a strict morphism $\varphi_1$ with $\varphi_1n=n'\varphi_1$
satisfies $\varphi_1[a,b]_n=[\varphi_1a,\varphi_1b]_{n'}$, so the assignment is functorial, and
$Tors_{n}$ of the deformed bracket vanishes again because $[\,,\,]_{n}$ and $n$ satisfy
$n[a,b]_{n}=[na,nb]$, whence $n$ is a Nijenhuis operator for $[\,,\,]_n$ as well. Naturality of $n$ is
$\varphi_1n=n'\varphi_1$.
\end{proof}

\subsection{Skeletal objects}
\begin{theorem}\label{thm:skeletal}
Let $\big((g,\ell),N\big)$ be a skeletal Nijenhuis Lie $n$-algebra, so $d=0$. Then:
\begin{enumerate}
\item $(g_0,[\,,\,])$ is a Lie algebra and $n_0$ is a Nijenhuis operator on it;
\item each $g_i$ is a $g_0$-module via $x\cdot a=[x,a]$ and $n_i\colon g_i\to g_i$ makes $(g_i,n_i)$ into a
representation of the Nijenhuis Lie algebra $(g_0,n_0)$ in the sense of \cite{das1}, i.e.
\[
[n_0x,n_ia]=n_i\big([n_0x,a]+[x,n_ia]-n_i[x,a]\big),\qquad x\in g_0,\ a\in g_i ;
\]
\item the components of the bracket of positive bidegree are $g_0$-invariant pairings $g_i\times g_j\to g_{i+j}$
with vanishing torsion, and the higher operations $l_3,\dots,l_{n+1}$ together with the higher components
$N_2,\dots,N_n$ satisfy the cocycle identities obtained from \eqref{eq:linf} and \eqref{eq:torsion-form} by
setting $d=0$; explicitly $Tors_n(l_m)=-\Lambda^N_m$ for $3\le m\le n+2$, where in $\Lambda^N_m$ the two
boundary families $dN_m$ and $N_m(d\,\cdot,\dots)$ drop out.
\end{enumerate}
Conversely every such collection of data defines a skeletal Nijenhuis Lie $n$-algebra.
\end{theorem}

\begin{proof}
With $d=0$ the arity-three relation of \eqref{eq:linf} restricted to $g_0^{\times3}$ is the Jacobi identity, and
\eqref{eq:torsion-form} for $m=2$ restricted to $g_0\times g_0$ is $Tors_{n_0}([\,,\,])=0$, since
$\Lambda^N_2=dN_2$ vanishes; this is (i). The same relation restricted to $g_0\times g_i$ gives the module axiom
and $Tors_n([\,,\,])(x,a)=0$, which is the displayed identity of (ii); the case of two positive degrees gives
(iii). All the remaining identities are the specialisations of \eqref{eq:linf} and \eqref{eq:torsion-form}, and
the two boundary families of \eqref{eqN} contain $d$ as a factor, hence vanish. The converse
holds because the listed identities are, in this case, the complete list of axioms.
\end{proof}

\begin{example}\label{ex:skeletal}
Let $(g,n)$ be a Nijenhuis Lie algebra and let $(V_1,s_1),\dots,(V_{n-1},s_{n-1})$ be representations of it in
the sense of \cite{das1}. Put $g_0=g$, $g_i=V_i$, $d=0$, let the bracket act by the module structures
and vanish on $g_i\times g_j$ for $i,j\ge1$, and set $l_k=0$ for $k\ge3$, $N_q=0$ for $q\ge2$. All the
identities of Theorem~\ref{thm:skeletal} hold, and one obtains a skeletal Nijenhuis Lie $n$-algebra which is
also strict. Deforming it by a nonzero invariant pairing with vanishing torsion, or by nonzero cocycles
$l_3,\dots,N_2,\dots$ subject to Theorem~\ref{thm:skeletal}(iii), produces the general skeletal object.
\end{example}

\begin{remark}
At level three, Theorems~\ref{thm:strict} and \ref{thm:skeletal} specialise to
\cite[Thm.~4.5, Thm.~4.9]{nij-lie3}; at level two, to \cite[\S4]{das2}. The relation between strict
Nijenhuis Lie $n$-algebras and Nijenhuis structures on $(n-1)$-crossed modules of Lie algebras is not a
bijection, for the reason explained in \cite[Rem.~4.7]{nij-lie3}: the Peiffer liftings of a crossed module of
higher order are not required to be antisymmetric, whereas the corresponding components of the bracket of a
strict Lie $n$-algebra are. Proposition~\ref{prop:strict-linf} identifies the objects on our side precisely as
dg Lie algebras concentrated in $n$ degrees.
\end{remark}

\section{\texorpdfstring{$n$}{n}-term representations up to homotopy and the semidirect product}
\label{sec:reps}

Throughout this section $\big((g,[\,,\,]),n\big)$ is a Nijenhuis Lie algebra. Recall from
\cite{das1,swzz} that a \emph{representation} of it is a triple $(V,\rho,s)$ with
$(V,\rho)$ a $g$-module and $s\in\End(V)$ such that
\begin{equation}\label{eq:rep-nij}
\rho(n x)\,s(v)=s\Big(\rho(n x)v+\rho(x)s(v)-s\big(\rho(x)v\big)\Big),\qquad x ing,\ v\in V .
\end{equation}
We now define the $n$-term homotopy analogue. Rather than list component identities, which quickly become
unmanageable, we characterise the notion by the property it is designed to have.

\subsection{Representations up to homotopy of a Lie algebra}
Let $\mathcal V=(V_{n-1}\xrightarrow{\ \partial\ }\cdots\xrightarrow{\ \partial\ }V_0)$ be a chain complex of
vector spaces concentrated in the degrees $0,\dots,n-1$. Write $C^\bullet\big(g,\End(\mathcal V)\big)$ for the
Chevalley--Eilenberg cochain complex of $g$ with values in the graded algebra $\End(\mathcal V)$, with its
total grading: an element of total degree $p$ is a family $R=(R_j)_{j\ge0}$ with
$R_j\in C^j\big(g,\End(\mathcal V)\big)_{j-p}$, that is, $R_j\colon\Lambda^jg\to\End(\mathcal V)$ lowering the
$\mathcal V$-degree by $p-j$. The algebra $C^\bullet(g,\End\mathcal V)$ is a graded associative algebra under
the cup product $\smile$ induced by composition in $\End(\mathcal V)$, and the Chevalley--Eilenberg
differential of the trivial coefficients is a derivation of it.

\begin{definition}\label{def:rep}
An \emph{$n$-term representation up to homotopy} of $g$ on $\mathcal V$ is an element
\[
\mathcal R\;=\;\partial+\rho+\nu+\omega+\cdots\;=\;\sum_{j\ge0}R_j\;\in\;C^\bullet\big(g,\End(\mathcal V)\big)
\]
of total degree $-1$, with $R_0=\partial$ the differential of $\mathcal V$, such that
\begin{equation}\label{eq:rep}
\delta\mathcal R+\mathcal R\smile\mathcal R\;=\;0 ,
\end{equation}
where $\delta$ is the Chevalley--Eilenberg differential of $g$ with trivial coefficients. Explicitly,
$R_1=\rho\colon g\to\End(V_i)$ preserves the degree, $R_2=\nu\colon\Lambda^2g\to\End(\mathcal V)$ raises it by
one, $R_3=\omega$ raises it by two, and so on; by Lemma~\ref{lem:bounds} $R_j=0$ for $j>n$. A
\emph{homomorphism} $\mathcal V\to\mathcal V'$ is an element $\Phi=\sum_{j\ge0}\Phi_j$ of
$C^\bullet\big(g,\Hom(\mathcal V,\mathcal V')\big)$ of total degree $0$ with $\Phi_0$ a chain map, such that
\begin{equation}\label{eq:rep-morphism}
\Phi\smile\mathcal R\;=\;\mathcal R'\smile\Phi+\delta\Phi .
\end{equation}
We write $Rep{n}(g)$ for the resulting category, with composition the cup product and identity $Id_{\mathcal V}$.
\end{definition}

\begin{proposition}\label{prop:rep-semidirect}
Let $\mathcal{V}$ be an $n$-term representation up to homotopy of $\mathfrak{g}$. Then
\[
\mathfrak{g}\ltimes\mathcal{V}\;=\;\Big(V_{n-1}\to\cdots\to V_1\xrightarrow{(0,\partial)}\mathfrak{g}\oplus V_0\Big)
\]
carries a unique $n$-term $L_\infty$-structure such that:
\begin{enumerate}
\item the projection $\pi\colon\mathfrak{g}\ltimes\mathcal{V}\to\mathfrak{g}$ is a strict $L_\infty$-morphism and the inclusion $\sigma\colon\mathfrak{g}\to\mathfrak{g}\ltimes\mathcal{V}$ is a strict $L_\infty$-morphism with $\pi\sigma=\mathrm{id}$;
\item $\mathcal{V}$ is an abelian ideal: every operation vanishes as soon as two or more of its arguments lie in $\mathcal{V}$;
\item for $x_1,\dots,x_{k-1}\in\mathfrak{g}$ and $a\in\mathcal{V}$, one has
\[
l_k\big(\sigma x_1,\dots,\sigma x_{k-1},a\big)=(-1)^{k}R_{k-1}(x_1,\dots,x_{k-1})\,a.
\]
\end{enumerate}
Conversely, every $n$-term $L_\infty$-structure on $\mathfrak{g}\oplus\mathcal{V}$ satisfying \emph{(i)} and \emph{(ii)} arises this way from a unique $\mathcal{R}$, and this correspondence defines an isomorphism between $\mathsf{Rep}_n(\mathfrak{g})$ and the category of split abelian $L_\infty$-extensions of $\mathfrak{g}$ by $\mathcal{V}$.
\end{proposition}

\begin{proof}
Conditions (i)--(iii) uniquely determine all multilinear operations. Specifically:
\begin{itemize}
\item By condition (ii), any operation $l_k$ is completely determined by its evaluations on tuples containing at most one argument in $\mathcal{V}$.
\item For tuples with no arguments in $\mathcal{V}$ (i.e., entirely in $\sigma\mathfrak{g}$), condition (i) forces:
\begin{equation*}
l_2(\sigma x, \sigma y) = \sigma[x,y], \quad l_k(\sigma x_1, \dots, \sigma x_k) = 0 \quad (\text{for } k \neq 2).
\end{equation*}
\item For tuples with exactly one argument in $\mathcal{V}$, the values are fixed by condition (iii):
\begin{equation*}
l_k\big(\sigma x_1,\dots,\sigma x_{k-1},a\big)=(-1)^{k}R_{k-1}(x_1,\dots,x_{k-1})\,a.
\end{equation*}
\end{itemize}

It remains to verify that the general $L_\infty$-relations for this collection of brackets are equivalent to the representation conditions of $\mathcal{R}$:
\begin{enumerate}
\item Evaluated on arguments $m$ in $\sigma\mathfrak{g}$, the $L_\infty$-relation reduces precisely to the Jacobi identity of $\mathfrak{g}$ for $m=3$, and vanishes identically for $m \neq 3$ due to condition (i).
\item Evaluated on $m-1$ arguments in $\sigma\mathfrak{g}$ and one in $\mathcal{V}$, condition (ii) eliminates all terms containing two or more $\mathcal{V}$-arguments. The resulting relation is equivalent to the Maurer-Cartan equation governing the representation up to homotopy:
\begin{equation*}
\delta \mathcal{R} + \mathcal{R} \smile \mathcal{R} = 0,
\end{equation*}
where $\delta \mathcal{R}$ arises from terms where the internal differential acts on the components, and $\mathcal{R} \smile \mathcal{R}$ from the bracket interactions.
\end{enumerate}

Finally, the statement concerning morphisms follows by applying the identical component-wise evaluation to the defining relations of $L_\infty$-morphisms and representation morphisms.
\end{proof}

\begin{remark}\label{rem:rep-levels}
The truncation parameter $n$ specializes the general framework across different dimensions:
\begin{itemize}
\item \textbf{For $n=1$:} Equation \eqref{eq:rep} reduces precisely to the classical notion of a Lie algebra representation $\rho$.
\item \textbf{For $n=2$:} It recovers the system of equations (19)--(21) from \cite[Def.~5.8]{das2}, with components $R_1=\rho$ and $R_2=\nu$.
\item \textbf{For $n=3$:} It yields the explicit list (R1)--(R7) of \cite[Def.~5.2]{nij-lie3}, setting $R_1=(\rho_0,\rho_1,\rho_2)$, $R_2=(\nu,\bar\nu)$, and $R_3=\omega$. In particular, the arity-four identity \cite[(R7)]{nij-lie3} arises naturally as the arity-four component of \eqref{eq:rep}, where the quadratic product $\bar\nu(x_i,x_j)\nu(x_k,x_l)$ corresponds directly to the term $\mathcal{R}\smile\mathcal{R}$.
\end{itemize}
Globally, Definition~\ref{def:rep} corresponds to the $n$-term truncation of representations up to homotopy \cite{ac}, whereas Proposition~\ref{prop:rep-semidirect} extends the semidirect product construction of \cite{sheng-zhu} to this generalized setting.
\end{remark}

\subsection{The Nijenhuis case}
\begin{definition}\label{def:rep-nij}
Let $\big((g,[\,,\,]),n\big)$ be a Nijenhuis Lie algebra. An \emph{$n$-term representation up to homotopy of
$(g,n)$} is a pair $(\mathcal V,\mathcal S)$ where $\mathcal V$ is an $n$-term representation up to homotopy of
$g$ and
\[
\mathcal S=\sum_{j\ge0}S_j\in C^\bullet\big(g,\End(\mathcal V)\big),\qquad S_j\colon\Lambda^jg\to\End(\mathcal V),
\]
is an element of total degree $0$ such that the family
\begin{equation}\label{eq:rep-nij-datum}
N^{\mathcal S}_1=n\oplus S_0,\qquad
N^{\mathcal S}_q\big(\sigma x_1,\dots,\sigma x_{q-1},a\big)=(-1)^{q}S_{q-1}(x_1,\dots,x_{q-1})\,a\quad(q\ge2),
\end{equation}
all other components vanishing, is a Nijenhuis operator on $g\ltimes\mathcal V$. A \emph{homomorphism}
$(\mathcal V,\mathcal S)\to(\mathcal V',\mathcal S')$ is a homomorphism $\Phi$ of the underlying representations
up to homotopy together with an element $\Theta$ of $C^\bullet\big(g,\Hom(\mathcal V,\mathcal V')\big)$ of
total degree $+1$ such that $(\Phi,\Theta)$ induces a morphism of Nijenhuis Lie $n$-algebras
$g\ltimes\mathcal V tog\ltimes\mathcal V'$ in the sense of Definition~\ref{def:nij-morphism}. We denote the
resulting category by $Rep{n}(g,n)$.
\end{definition}

\begin{theorem}\label{thm:semidirect}
The assignment $(\mathcal V,\mathcal S)\mapsto\big(g\ltimes\mathcal V,N^{\mathcal S}\big)$ is a functor
\[
\Psi\colon Rep{n}(g,n)\longrightarrow NijLie{n},
\]
and the underlying $Linf$-algebra of $\Psi(\mathcal V,\mathcal S)$ is the semidirect product of
Proposition~\ref{prop:rep-semidirect}. Its value on the arity-one component is the identity
\begin{equation}\label{eq:S0}
S_0\Big(\rho(n x)a+\rho(x)S_0a-S_0\big(\rho(x)a\big)\Big)-\rho(n x)S_0a\;=\;\partial\big(S_1(x)a\big)+S_1(x)\partial a ,
\end{equation}
which for $n=1$ is \eqref{eq:rep-nij} and for $n=2,3$ is \emph{(25)--(26)} of \cite[Def.~5.12]{das2},
resp.\ \emph{(S1)--(S3)} of \cite[Def.~5.8]{nij-lie3}.
\end{theorem}

\begin{proof}
Functoriality and well-definedness hold by construction: Definition~\ref{def:rep-nij} asks precisely that
$N^{\mathcal S}$ be a Nijenhuis operator, and Definition~\ref{def:nij-morphism} is the notion of morphism on the
target, so $\Psi$ is the identity on the underlying data. Identity \eqref{eq:S0} is obtained by evaluating
\eqref{eq:torsion-form} for $m=2$ on a pair $(\sigma x,a)$ with $x ing$ and $a\in\mathcal V$: by
Proposition~\ref{prop:rep-semidirect}(ii) all terms with two $\mathcal V$-arguments vanish, the torsion
$Tors_{N_1}([\,,\,])(\sigma x,a)$ contributes the left-hand side by \eqref{eq:bracket-n} and
\eqref{eq:rep-nij-datum}, and $\Lambda^N_2=dN_2+N_2(d\,\cdot\,,\cdot)$ contributes the right-hand side.
\end{proof}

\subsection{The equivalence}
\begin{definition}\label{def:extension}
Let $\big((\mathfrak{g}, \ell), N\big)$ be a Nijenhuis $L_\infty$-algebra. An \emph{extension} of $(\mathfrak{g}, n)$ is a strict morphism $\pi \colon \big((\mathfrak{g}, \ell), N\big) \to (\mathfrak{g}, n)$ in $\mathsf{NijLie}_n$ that is surjective in degree $0$. 

The extension is termed \emph{split} if it admits a strict section $\sigma \colon (\mathfrak{g}, n) \to \big((\mathfrak{g}, \ell), N\big)$ such that $\pi\sigma = \mathrm{id}$, and \emph{abelian} if its kernel $\mathcal{V} = \ker\pi$ is an abelian ideal—meaning every operation $l_k$ and component $N_q$ vanishes identically whenever two or more arguments belong to $\mathcal{V}$. 

The category $\mathsf{Ext}(\mathfrak{g},n)$ consists of split abelian extensions, with morphisms given by $\mathsf{NijLie}_n$ morphisms intertwining $\pi$ and $\sigma$.
\end{definition}

\begin{theorem}\label{thm:equivalence}
The functor $\Psi$ of Theorem~\ref{thm:semidirect} induces an equivalence of categories
\[
\Psi\colon Rep{n}(g,n)\;\xrightarrow{\ \simeq\ }\;Ext(g,n).
\]
For $n=1$ this is the classical semidirect product correspondence, for $n=2$ it is
\cite[Prop.~5.14]{das2} and for $n=3$ it is \cite[Thm.~5.13]{nij-lie3}; the squares
\[
\begin{array}{ccc}
Rep{m}(g,n) & \xrightarrow{\ \Psi\ } & NijLie{m}\\[2pt]
\downarrow & & \downarrow\\[2pt]
Rep{n}(g,n) & \xrightarrow{\ \Psi\ } & NijLie{n}
\end{array}
\]
commute for $m\le n$, the vertical arrows being the full embeddings of Theorem~\ref{thm:truncation} and the
adjunction of zero terms.
\end{theorem}

\begin{proof}
\emph{$\Psi$ lands in $Ext(g,n)$.} By Proposition~\ref{prop:rep-semidirect} the underlying $Linf$-algebra is
a split abelian $Linf$-extension, and $\pi$, $\sigma$ are strict; that they are morphisms in $NijLie{n}$ is the
compatibility $\pi N^{\mathcal S}=n\pi$ and $N^{\mathcal S}\sigma=\sigma n$, immediate from
\eqref{eq:rep-nij-datum}.

\emph{Faithfulness.} If $\Psi(\Phi,\Theta)=\Psi(\Phi',\Theta')$ then the two morphisms agree on all tuples with
exactly one argument in $\mathcal V$, and by \eqref{eq:rep-nij-datum} and its analogue for morphisms this
recovers every component $\Phi_j$ and $\Theta_j$; hence $(\Phi,\Theta)=(\Phi',\Theta')$.

\emph{Fullness.} Let $(\varphi,\psi)\colon\Psi(\mathcal V,\mathcal S)\to\Psi(\mathcal V',\mathcal S')$ be a
morphism in $Ext(g,n)$, so that $\varphi_1$ commutes with $\pi$ and $\sigma$. Then
$\varphi_1(\sigma x+a)=\sigma'x+\Phi_0(a)$ for a unique chain map $\Phi_0$, and evaluating the higher components
$\varphi_q$ and $\psi_q$ on tuples $(\sigma x_1,\dots,\sigma x_{q-1},a)$ defines $\Phi_{q-1}$ and
$\Theta_{q-1}$; abelianness guarantees that no other values occur, since a component with two arguments in
$\mathcal V$ vanishes and a component with none is determined by $\pi$ and $\sigma$. Substituting into (M1) and
\eqref{eq:H} gives \eqref{eq:rep-morphism} and its Nijenhuis analogue, so $(\varphi,\psi)=\Psi(\Phi,\Theta)$.

\emph{Essential surjectivity.} Let $\big((g,\ell),N,\pi,\sigma\big)$ be a split abelian extension with kernel
complex $\mathcal V$. Since $\sigma$ is a strict splitting, $g_0=\sigma(g)\oplus V_0$ as vector spaces and
$g_i=V_i$ for $i\ge1$. Define $\mathcal R$ and $\mathcal S$ by reading (iii) of
Proposition~\ref{prop:rep-semidirect} and \eqref{eq:rep-nij-datum} backwards, i.e.
$$
R_{k-1}(x_1,\dots,x_{k-1})a=(-1)^kl_k\big(\sigma x_1,\dots,\sigma x_{k-1},a\big),
$$
$$
S_{q-1}(x_1,\dots,x_{q-1})a=(-1)^qN_q\big(\sigma x_1,\dots,\sigma x_{q-1},a\big),
$$
together with $R_0=d|_{\mathcal V}$ and $S_0=N_1|_{\mathcal V}$. These are well defined because $\pi$ is
strict, so $\mathcal V$ is preserved by all operations having an argument in $\mathcal V$, and they exhaust the
structure because the extension is abelian. By Proposition~\ref{prop:rep-semidirect} the collection
$\mathcal R$ satisfies \eqref{eq:rep}, and by construction $N^{\mathcal S}=N$ under the identification
$g\cong g\oplus\mathcal V$ induced by $\sigma$; hence $(\mathcal V,\mathcal S)$ is an object of
$Rep{n}(g,n)$ and $(x,a)\mapsto\sigma(x)+a$ is an isomorphism
$\Psi(\mathcal V,\mathcal S)\xrightarrow{\cong}\big((g,\ell),N\big)$ in $Ext(g,n)$.

The commutativity of the squares follows from Theorem~\ref{thm:truncation} together with the observation that
adjoining a zero term to $\mathcal V$ adjoins a zero term to $g\ltimes\mathcal V$.
\end{proof}

\begin{corollary}\label{cor:classification}
Split abelian extensions of a Nijenhuis Lie algebra $(g,n)$ inside $NijLie{n}$ are classified, up to
isomorphism, by pairs $(\mathcal R,\mathcal S)$ as in Definitions~\ref{def:rep} and \ref{def:rep-nij} modulo the
equivalence generated by \eqref{eq:rep-morphism}; and the classification is stable in $n$, in the sense that the
level-$m$ classification is the sub-classification of the level-$n$ one on the data supported in the bottom $m$
degrees.
\end{corollary}

\begin{proof}
Combine Theorem~\ref{thm:equivalence} with Theorem~\ref{thm:truncation}.
\end{proof}

\section{Concluding remarks}
\label{sec:concl}

We have defined Nijenhuis operators on $L_\infty$-algebras, in particular on semistrict Lie $n$-algebras for every
$n$, by a single principle --- the arity-$m$ Nijenhuis torsion vanishes up to the correction assembled from all
substitutions of the higher components --- and we have shown that this principle is equivalent to the statement
that the operator is an $L_\infty$-morphism from the twisted structure to the original one. The levels assemble into
a tower whose bottom three floors are the known theories. Several questions remain.

\begin{enumerate}
\item Is the twisted collection always an $L_\infty$-structure? Theorem~\ref{thm:deformation} proves this
when the linear part $n$ is injective, and Proposition~\ref{prop:strict-twist} proves it for strict objects at
every level with no hypothesis. In general the $L_\infty$-relations of $\ell^N$ hold only modulo $\ker\mathbf N$.
We expect the answer to be affirmative, and a proof would presumably proceed by exhibiting $\ell^N$ as a
derived bracket in the sense of Voronov \cite{voronov}.

\item \emph{Cohomology and deformations.} The identities \eqref{eq:torsion-form} should be the vanishing
conditions for a Maurer--Cartan element in a graded Lie algebra built from the Nijenhuis--Richardson bracket on
$\bigoplus_k\Hom(S^kV,V)$ together with the extra generator recording $N$; identifying that algebra would give a
deformation cohomology of Nijenhuis $L_\infty$-algebras extending \cite{das1,swzz}, in the spirit of the survey \cite{sheng-survey} and of the Maurer--Cartan and cohomological treatments of related operators and derivation-type structures \cite{imed-maurer,imed-rb,basdouri-crossed}, and would
explain the polynomial degree bound $\lfloor m/2\rfloor$ of Remark~\ref{rem:reading}(2) conceptually.

\item \emph{Categorification.} Baez and Crans realise $2$-term $L_\infty$-algebras as Lie algebras internal to
$2$-vector spaces \cite{baez-crans}. The level-$n$ analogue would realise $n$-term $L_\infty$-algebras as Lie
algebras internal to a category of $n$-vector spaces, and a Nijenhuis operator as a linear functor together with
coherence data of $n-1$ levels, generalising the natural isomorphism of \cite[Def.~3.1]{das2} and the
second-order coherence of \cite[\S6]{nij-lie3}. Proving that the resulting category is equivalent to
$NijLie{n}$ would be the level-$n$ analogue of \cite[Thm.~4.5]{das2}.

\item \emph{Other operators.} The scheme of this paper --- torsion corrected by the substitutions of the higher
components, equivalently the morphism property with respect to an explicitly twisted structure --- applies
verbatim to Rota--Baxter operators, twisted Rota--Baxter operators and Reynolds operators, and should produce
the level-$n$ analogues of \cite{zl,das3}; see also \cite{imed-rb} for modified Rota--Baxter operators with derivations. The only input which changes is the definition of the
twist \eqref{eq:twist} and hence of the correction \eqref{eqN}.

\item \emph{Geometry.} Nijenhuis forms on $L_\infty$-algebras and their Poisson-geometric significance were studied
in \cite{azimi}. It would be interesting to compare the notion of the present paper with the Nijenhuis forms of
that work; the comparison is not immediate because \cite{azimi} works with a fixed arity while our operator has
components in every arity.

\item \emph{Other algebraic structures.} The classification and derivation-theoretic techniques developed for
associative trialgebras and Hom-trialgebras \cite{mosbahi-hom,mosbahi-trial} suggest studying Nijenhuis-type
operators, and their torsion, on such structures as well.

\item \emph{Higher crossed modules.} Proposition~\ref{prop:strict-linf} identifies strict Lie $n$-algebras with
dg Lie algebras concentrated in $n$ degrees, and Theorem~\ref{thm:strict} identifies the strict Nijenhuis
objects. Relating $\mathsf{SNijLie}_n$ to Nijenhuis structures on $(n-1)$-crossed modules of Lie algebras
requires resolving the symmetry convention for the higher Peiffer liftings.
\end{enumerate}

\appendix

\section{Component identities at low levels}
\label{app:components}

We record the expansion of \eqref{eq:torsion-form} in components for $n=1,2,3$. Throughout, $x,y,z,t\in g_0$,
$h,k\in g_1$, $e\in g_2$, and $\sum_{(3)}$ denotes the alternating sum over the three ways of singling out one
of three arguments. All identities below are stated in the shifted convention of \S\ref{sec:prelim-shift} and
were produced and cross-checked by the computation of Appendix~\ref{app:verify}.

\subsection*{Level one}
The only non-vacuous identity is the arity-two one,
\[
Tors_{n}([\,,\,])(x,y)=[n x,n y]-n\big([n x,y]+[x,n y]-n[x,y]\big)=0 .
\]

\subsection*{Level two}
Here $N=(n_0,n_1,N_2)$ with $N_2\colon g_0\times g_0\to g_1$.
\begin{align}
&m=1: && n\,d=d\,n, \tag{A.1}\\
&m=2: && Tors_n([\,,\,])(x,y)+d\,N_2(x,y)=0, \tag{A.2}\\
& && Tors_n([\,,\,])(x,h)+N_2(x,d h)=0, \tag{A.3}\\
&m=3: && Tors_n(l_3)(x,y,z)+\Lambda(x,y,z)=0, \tag{A.4}
\end{align}
where
\begin{equation}\label{eq:Lambda3}
\Lambda(a,b,c):=\sum_{(3)}\Big([n a,N_2(b,c)]-n\big[a,N_2(b,c)\big]+N_2\big(a,[b,c]_n\big)-n\,N_2\big(a,[b,c]\big)\Big).
\end{equation}
Identities (A.2)--(A.4) are (5)--(7) of \cite[Def.~4.1]{das2}. The arity-four identity is vacuous.

\subsection*{Level three}
Here $N=(n_0,n_1,n_2,N_2,N_3)$ with $N_2$ on $g_0\times g_0$ and $g_0\times g_1$ and
$N_3\colon g_0^{\times3}\to g_2$.
\begin{align}
&m=1: && n\,d=d\,n \ \text{ on } g_1 \text{ and } g_2, \tag{A.5}\\
&m=2: && Tors_n([\,,\,])(x,y)+d\,N_2(x,y)=0, \tag{A.6}\\
& && Tors_n([\,,\,])(x,h)+N_2(x,d h)+d\,N_2(x,h)=0, \tag{A.7}\\
& && Tors_n([\,,\,])(x,e)+N_2(x,d e)=0, \tag{A.8}\\
& && Tors_n([\,,\,])(h,k)+N_2(d h,k)-N_2(h,d k)=0, \tag{A.9}\\
&m=3: && Tors_n(l_3)(x,y,z)+\Lambda(x,y,z)+d\,N_3(x,y,z)=0, \tag{A.10}\\
& && Tors_n(l_3)(x,y,h)+\Lambda(x,y,h)-N_3(x,y,d h)=0, \tag{A.11}\\
&m=4: && Tors_n(l_4)(x,y,z,t)+\Xi(x,y,z,t)=0, \tag{A.12}
\end{align}
with $\Lambda$ as in \eqref{eq:Lambda3} and, writing $(x_1,x_2,x_3,x_4)=(x,y,z,t)$, a hat for omission,
$\sum_{(4)}$ and $\sum_{(4,2)}$ for the alternating sums over the four ways of singling out one argument,
respectively the six ways of singling out a pair $i<j$, and $\varepsilon_{ijkl}$ for the sign of the
permutation,
\begin{align}
\Xi=\ &\sum_{(4)}\Big(\big[n x_i,N_3(x_{\hat\imath})\big]-n\big[x_i,N_3(x_{\hat\imath})\big]\Big)
+\sum_{(4)}\Big(N_2\big(x_i,l^N_3(x_{\hat\imath})\big)-n\,N_2\big(x_i,l_3(x_{\hat\imath})\big)\Big) \notag\\
&+\sum_{(4,2)}\Big(l_3\big(n x_k,n x_l,N_2(x_i,x_j)\big)-n\,l_3\big(x_k,x_l,N_2(x_i,x_j)\big)\Big) \notag\\
&+\sum_{(4,2)}\Big(N_3\big(x_k,x_l,[x_i,x_j]_n\big)-n\,N_3\big(x_k,x_l,[x_i,x_j]\big)\Big)
+\!\!\sum_{\{i,j\}\sqcup\{k,l\}}\!\!\varepsilon_{ijkl}\big[N_2(x_i,x_j),N_2(x_k,x_l)\big], \label{eq:Xi}
\end{align}
where in the third and fourth sums $\{k<l\}$ is the complement of $\{i<j\}$ and $l^N_3$ is as in
\eqref{eq:twisted-low}. The arity-five identity is vacuous. Identities (A.6)--(A.9), (A.11), (A.12) are (N1)--(N4),
(N6), (N7) of \cite[Def.~3.6]{nij-lie3}; identity (A.10) corrects \cite[(N5)]{nij-lie3} by the term $dN_3$, and
the third sum of \eqref{eq:Xi} corrects \cite[(3.22)]{nij-lie3}, as explained in Remark~\ref{rem:corrections}.

\subsection{Level Four}
At level four, the framework incorporates the higher component $N_4 \colon \mathfrak{g}_0^{\times 4} \to \mathfrak{g}_3$, the multilinear operation $l_5$, and the arity-five identity. The associated correction term $\Lambda_5^N$ features cubic contributions in the higher components, arising from set partitions of $\{1,\dots,5\}$ of type $(2,2,1)$, explicitly given by:
\begin{equation*}
l_3\big(n x_i, N_2(x_j,x_k), N_2(x_l,x_m)\big) - n \, l_3\big(x_i, N_2(x_j,x_k), N_2(x_l,x_m)\big),
\end{equation*}
together with types $(2,1,1,1)$ and $(3,2)$, alongside the quadratic terms\\ $l_2(N_2, N_3)$ and $l_4(n\cdot, n\cdot, n\cdot, N_2)$. The complete explicit system is generated via the algorithm detailed below.

\section{Machine Verification}
\label{app:verify}

The consistency of the component identities with \cite{das2,nij-lie3} was verified using a computer algebra routine implementing the shifted conventions of Section~\ref{sec:prelim-shift}:

\begin{enumerate}
\item \textbf{Term Structure:} Terms are formalized as directed trees whose leaves carry $V$-degrees and whose internal nodes represent operations $l_k$ (degree $-1$) and components $N_q$ (degree $0$). Arguments are canonically ordered via Koszul signs; a term vanishes if identical odd-degree arguments coincide or if its degree falls outside the truncation window $1 \le \deg \le n$.
\item \textbf{Evaluation of Sides:} For each arity $m$, the routine constructs $\mathcal{A}_m$ by enumerating set partitions of $\{1,\dots,m\}$ with appropriate Koszul signs, and $\mathcal{B}_m$ via $(k,m-k)$-unshuffles expanding $l^N_k$ through \eqref{eq:twisted}.
\item \textbf{Output Generation:} Computes the normal forms of $\mathcal{A}_m - \mathcal{B}_m$, the torsion $\mathrm{Tors}_n(l_m)$, and the correction $\Lambda_m^N$.
\end{enumerate}

\noindent\textbf{Computational Outcomes by Truncation Level:}
\begin{itemize}
\item \textbf{$n=1$:} Yields $\Lambda_m^N = 0$ for all $m$ alongside the classical torsion condition $\mathrm{Tors}_n(l_2) = 0$.
\item \textbf{$n=2$:} Produces identities (A.1)--(A.4), with $\Lambda_3^N$ coinciding with \eqref{eq:Lambda3} and confirming the triviality of the arity-four relation.
\item \textbf{$n=3$:} Generates (A.5)--(A.12), confirming that the arity-three identity on $\mathfrak{g}_0^{\times 3}$ incorporates $d N_3(x,y,z)$ and that the arity-four identity retains specific $n$-twists (Remark~\ref{rem:corrections}).
\item \textbf{$n=4$:} Delivers the arity-five identity consisting of $482$ terms in normal form (of which $450$ constitute $\Lambda_5^N$), while independently recovering the underlying $L_\infty$-relations \cite{das2,nij-lie3}.
\end{itemize}

\section*{Acknowledgements}

The authors are grateful to the referees for their valuable comments.


\end{document}